\documentclass[11pt]{amsart}
\usepackage{color}
\usepackage[capitalize,nameinlink,noabbrev,nosort]{cleveref}
\usepackage{graphicx, enumitem}
\usepackage{tikz}

\usepackage[leqno]{amsmath}
\makeatletter
\newcommand{\leqnomode}{\tagsleft@true}
\newcommand{\reqnomode}{\tagsleft@false}
\makeatother

\makeatletter
\newcommand{\customlabel}[2]{
\protected@write \@auxout {}{\string \newlabel {#1}{{#2}{}}}}
\makeatother

\newtheorem{thm}{Theorem}[section]
\newtheorem{theorem}[thm]{Theorem}
\newtheorem{lemma}[thm]{Lemma}
\newtheorem{prop}[thm]{Proposition}
\newtheorem{proposition}[thm]{Proposition}
\newtheorem{cor}[thm]{Corollary}

\theoremstyle{definition}
\newtheorem{defn}[thm]{Definition}
\newtheorem{definition}[thm]{Definition}
\newtheorem{remark}[thm]{Remark}
\newtheorem{example}[thm]{Example}

\DeclareMathOperator{\wt}{wt}

\DeclareMathOperator{\PBT}{PBT}
\DeclareMathOperator{\rev}{rev}
\DeclareMathOperator{\inc}{inc}
\DeclareMathOperator{\Tab}{Tab}

\DeclareMathOperator{\id}{id}

\DeclareMathOperator{\dec}{dec}

\DeclareMathOperator{\MLQ}{MLQ}

\DeclareMathOperator{\ASEP}{ASEP}
\DeclareMathOperator{\skipped}{skipped}

\DeclareMathOperator{\free}{free}

\DeclareMathOperator{\QT}{QT}

\DeclareMathOperator{\maj}{maj}

\DeclareMathOperator{\leg}{leg}
\DeclareMathOperator{\arm}{arm}
\DeclareMathOperator{\coinv}{coinv}

\newcommand{\Q}{\mathbb Q}
\newcommand{\Z}{\mathbb Z}

\newcommand{\x}{{\mathbf x}}

\newcommand\cell[3]{
\def\i{#1} \def\j{#2} \def\entry{#3}

\draw (\j-1,-\i)--(\j,-\i)--(\j,-\i+1)--(\j-1,-\i+1)--(\j-1,-\i);
\node at (\j-.5,-\i+.5) {\entry};
}

\newcommand\qtrip[3]{
\begin{tikzpicture}[scale=0.5]
\cell{1}{0}{#1}; \cell{2}{0}{#2}; \cell{2}{2.7}{#3};
\node at (1,-1.5) {$\cdots$};
\end{tikzpicture}
}

\newcommand\qqua[4]{
\begin{tikzpicture}[scale=0.5]
\cell{1}{0}{#1} \cell{2}{0}{#2} \cell{1}{2.7}{#3} \cell{2}{2.7}{#4}
\node at (1,-1.5) {$\cdots$};
\end{tikzpicture}
}

\title{From multiline queues to Macdonald polynomials via the exclusion process}
\date{\today}
\author{Sylvie Corteel}
\address{Department of Mathematics, UC Berkeley, Berkeley CA}
\email{corteel@berkeley.edu}
\author{Olya Mandelshtam}
\address{Combinatorics and Optimization,
University of Waterloo, Waterloo, ON}
\email{omandels@uwaterloo.ca}
\author{Lauren Williams}
\address{Department of Mathematics,
Harvard University, Cambridge, MA}
\email{williams@math.harvard.edu}
\thanks{SC was in residence at MSRI in Berkeley (NSF grant DMS-1440140) 
and was funded by the Miller
Institute, Berkeley during the elaboration of this work. SC is partially funded by NSF grant DMS-2054482.  OM was supported by NSF grant DMS-1704874.
LW was partially supported by NSF grant DMS-1600447.}

\begin{document}


\begin{abstract}
Recently James Martin \cite{Martin} introduced \emph{multiline queues}, and used them to give a combinatorial formula for the stationary distribution of the multispecies asymmetric simple  exclusion process (ASEP) on a circle. The ASEP is a model of particles hopping on a one-dimensional lattice, which was introduced around 1970 \cite{bio, Spitzer}, and has been extensively studied in statistical mechanics, probability, and combinatorics. In this article we give an independent proof of Martin's result, and we show that by introducing additional statistics on multiline queues, we can use them to give a new combinatorial formula for both the symmetric Macdonald polynomials $P_{\lambda}(\mathbf{x}; q, t)$, and the nonsymmetric Macdonald polynomials $E_{\lambda}(\mathbf{x}; q, t)$, where $\lambda$ is a partition. This formula is rather different from others that have appeared in the literature \cite{HHL2}, \cite{RamYip}, \cite{Lenart}. Our proof uses results of Cantini, de Gier, and Wheeler \cite{CGW}, which recently linked the multispecies ASEP on a circle to Macdonald polynomials.
\end{abstract}

\maketitle
\setcounter{tocdepth}{1}
\tableofcontents

\section{Introduction and results}

Introduced in the late 1960's \cite{bio, Spitzer},
the \emph{asymmetric simple exclusion process} (ASEP) is a model
of interacting
particles hopping left and right on a one-dimensional lattice of $n$
sites. There are many versions of the ASEP: the lattice might be a lattice with open boundaries,
or a ring, among others; and we may allow multiple species of particles with different ``weights".
In this article, we will be concerned with the multispecies ASEP on a ring, 
where the rate of two adjacent particles swapping places is either $1$ or $t$, depending
on their relative weights.
Recently James Martin \cite{Martin} gave a combinatorial formula in terms of 
\emph{multiline queues} for the stationary distribution 
of this multispecies ASEP on a ring, building on his earlier joint work 
with Ferrari \cite{FerrariMartin}.

On the other hand, recent work of 
Cantini, de Gier, and Wheeler 
\cite{CGW} gave a link between the multispecies ASEP on a ring and \emph{Macdonald polynomials}.
Symmetric Macdonald polynomials $P_{\lambda}(\mathbf{x}; q, t)$ \cite{Macdonald} are a family of multivariable orthogonal polynomials
indexed by partitions, whose coefficients depend on two parameters $q$ and $t$; they generalize
multiple important families of polynomials, 
including Schur polynomials (at $q=t$, or equivalently, at $q=t=0$)  and
Hall-Littlewood polynomials (at $q=0$).
\emph{Nonsymmetric Macdonald polynomials} \cite{Cher1, MacdonaldBourbaki} were introduced shortly 
after the introduction of Macdonald polynomials, and defined in terms of 
\emph{Cherednik operators}; the symmetric Macdonald polynomials can be constructed from 
their nonsymmetric counterparts.

There has been a lot of work devoted to understanding Macdonald polynomials from a combinatorial
point of view.
Haglund-Haiman-Loehr \cite{HHL2, HHL1} gave a combinatorial formula for the 
\emph{transformed Macdonald polynomials} $\tilde{H}_{\mu}(\mathbf{x}; q, t)$
(which are connected to the 
 geometry of the Hilbert scheme
\cite{HaimanHilbert}) as well as for the \emph{integral forms}
$J_{\mu}(\mathbf{x}; q, t)$, which are scalar multiples of the classical monic
forms $P_{\mu}(\mathbf{x}; q, t)$.  
They also 
 gave a formula for the nonsymmetric  Macdonald polynomials
\cite{HHL3}.  Building on work of Schwer \cite{Schwer}, 
Ram and Yip \cite{RamYip} gave general-type 
formulas for both the Macdonald polynomials 
$P_{\lambda}(\mathbf{x}; q, t)$ and the nonsymmetric Macdonald polynomials; 
however, their type $A$ formulas have many terms.  
Lenart \cite{Lenart} showed how to ``compress" the Ram-Yip formula in type 
A to obtain a Haglund-Haiman-Loehr type formula for the polynomials
$P_{\lambda}(\mathbf{x}; q, t)$.  (However, for technical reasons,
his paper only treats the case where 
 $\lambda$ is regular, i.e. the parts of $\lambda$ are distinct.)
Finally,  Ferreira \cite{thesis}  
and Alexandersson \cite{A} gave Haglund-Haiman-Loehr type formulas for 
\emph{permuted basement Macdonald polynomials}, which generalize the 
nonsymmetric Macdonald polynomials.

The main goal of this article is to 
define some polynomials combinatorially 
in terms of multiline queues which simultaneously
compute the stationary distribution of the multispecies ASEP and also 
symmetric Macdonald 
polynomials $P_{\lambda}(\mathbf{x}; q, t)$.  More specifically, we introduce some polynomials
$F_{\mu}(x_1,\dots,x_n; q, t) = F_{\mu}(\mathbf{x};q,t) \in \Z[x_1,\dots,x_n](q,t)$ 
which are certain weight-generating functions
for multiline queues with bottom row $\mu$, where $\mu=(\mu_1,\dots,\mu_n)$ is an arbitrary weak composition.  
We show that these polynomials have the following properties:
\begin{enumerate}
\item \label{one} When $x_1=\dots=x_n=1$ and $q=1$, $F_{\mu}(\mathbf{x}; q, t)$ is proportional to the steady state probability that the multispecies ASEP is in state $\mu$.
(This recovers a result of Martin \cite{Martin}, but our proof is independent of his.) 
\item \label{two} 
When $\mu$ is a partition, $F_{\mu}(\mathbf{x}; q, t)$ is equal to the nonsymmetric Macdonald
polynomial $E_{\mu}(\mathbf{x}; q, t)$. 
\item \label{three} 
For any partition $\lambda$, the quantity $Z_{\lambda}(\mathbf{x}; q, t) := 
		\sum_{\mu} F_{\mu}(\mathbf{x}; q, t)$ 
		(where the sum is over all distinct compositions obtained by permuting the parts of 
$\lambda$) 
		is equal to the symmetric Macdonald polynomial $P_{\lambda}(\mathbf{x}; q, t)$.
\end{enumerate}  

In the remainder of the introduction 
we will make the above statements more precise.

\subsection{The multispecies ASEP}

We start by defining the multispecies ASEP or the $L$-ASEP as a Markov chain on the 
cycle $\Z_n$ with $L$ classes of particles as well as holes. The $L$-ASEP on a ring is a natural generalization for the two-species ASEP; for the latter, 
solutions were given using a matrix product formulation in terms of a quadratic algebra similar to the matrix ansatz described in \cite{DEHP}. 

For the $L$-ASEP when $t=0$ (i.e. particles only hop in one direction), 
Ferrari and Martin \cite{FerrariMartin} proposed 
a combinatorial solution for the stationary distribution using multiline queues. This construction was restated as a matrix product solution in 
\cite{EvansFerrariMallick08} and was generalized to 
the partially asymmetric case ($t$ generic) in 
\cite{ProlhacEvansMallick09}.
In \cite{AritaAyyerMallickProlhac12} 
the authors explained how to construct an explicit 
representation of the algebras involved in the $L$-ASEP.
Finally James Martin \cite{Martin}
gave an ingenious combinatorial solution
for the stationary distribution of the $L$-ASEP when $t$ is generic,
using more general multiline queues and building on ideas from  \cite{FerrariMartin} and \cite{EvansFerrariMallick08}. 

\begin{definition}
Let $\lambda=(\lambda_1\geq\cdots\geq\lambda_n\geq 0)$ be a partition.  We
let  $S_n(\lambda)$ denote the set of all distinct weak compositions $\mu$
	obtained by permuting the parts of $\lambda$.
\end{definition}
For example, if $\lambda = (2,2,1)$, then 
$S_3(\lambda) = \{(2,2,1), (2,1,2), (1,2,2)\}$.

\begin{definition}\label{def:ASEP}
Let $\lambda = (\lambda_1 \geq \lambda_2 \geq \cdots \geq \lambda_n \geq 0)$
be a partition with greatest part $\lambda_1 = L$, 
        and let $t$ be a constant such that $0 \leq t \leq 1$.
Our state space will be $S_n(\lambda)$; 
note that we consider indices of $\mu\in S_n(\lambda)$ 
modulo $n$; i.e.
if $\mu=\mu_1\ldots \mu_n$ is a composition, then $\mu_{n+1}=\mu_1$.
The \emph{multispecies asymmetric simple exclusion process}
        $\ASEP(\lambda)$ on a ring
        is the Markov chain on $S_n(\lambda)$
        with transition probabilities:
\begin{itemize}
	\item If $\mu = A i j B$ and $\nu = A j i B$ are in $S_n(\lambda)$
		(here $A$ and $B$ are words in the parts of $\lambda$), 
then 
$P_{\mu,\nu} = \frac{t}{n}$ if $i>j$ and $P_{\mu,\nu} = \frac{1}{n}$ if $i<j$.
\item Otherwise $P_{\mu,\nu} = 0$ for $\nu \neq \mu$ and
$P_{\mu,\mu} = 1-\sum_{\mu \neq \nu} P_{\mu,\nu}$.
\end{itemize}
We think of the $1$'s, $2$'s, \dots, $L$'s as representing various types of
particles of different weights;
each $0$ denotes an empty site.
See \cref{parameters1}.
\end{definition}

\begin{figure}[!ht]
 \centerline{\includegraphics[height=1in]{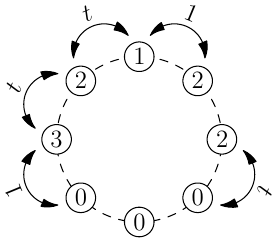}}
\centering
 \caption{A state in the multispecies ASEP on the lattice $\Z_8$. There is 
one particle of type $3$, three particles of type $2$, one particle of type $1$, and three holes, so we refer to this Markov chain as $\ASEP(3, 2,2,2,1,0,0,0)$. The rates $1$ and $t$ represent probabilities $1/8$ and $t/8$ respectively of swapping the corresponding particles.}
\label{parameters1}
 \end{figure}

 
 \begin{remark}
Note that in the literature on the ASEP, the hopping rate is often denoted by $q$.  We are using $t$ here instead in order to be consistent with the notation of \cite{CGW-arxiv, CGW}, and to make contact with the literature on Macdonald polynomials. Furthermore, the convention used in \cite{FerrariMartin, Martin} swaps the roles of 1 and $t$ in our \cref{def:ASEP}.  
\end{remark}

\subsection{Multiline queues}
\label{1point2}

We now define ball systems and multiline queues. These concepts are due to Ferrari and Martin \cite{FerrariMartin} 
for the case $t=0$ and $q=1$ and to Martin \cite{Martin} for the case $t$ general and $q=1$.
\begin{definition}\label{def:MLQ}
Fix positive integers $L$ and $n$.
A \emph{ball system} $B$ 
is an $L \times n$ array
in which each of the $Ln$ positions is either empty or occupied by a ball.
We number the rows from bottom to top from $1$ to $L$,
and the columns from left to right from $1$ to $n$.  Moreover we require
that there is at least one ball in the top row,
and that the
 number of balls in each row is weakly increasing from top to bottom.
See \cref{parameters2} for an example.
\end{definition}
\begin{figure}[!ht]
  \centerline{\includegraphics[height=1in]{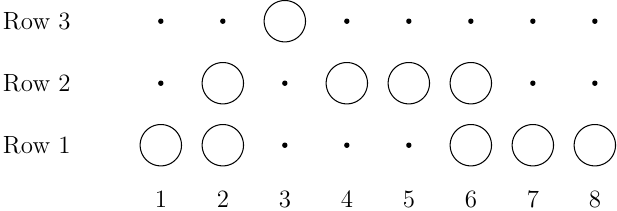}}
\centering
 \caption{
A ball system.}
\label{parameters2}
 \end{figure}

\begin{definition}
Given an $L \times n$ ball system $B$, a multiline queue $Q$ (for $B$)
is, for each row $r$ where $2 \leq r \leq L$, a matching of balls
from row $r$ to row $r-1$.  A ball $b$ may be matched to any ball $b'$
in the row below it; we connect $b$ and $b'$ by a shortest strand that 
travels either straight down or 
from left to right (allowing the strand to wrap around the cylinder
if necessary).
Here the balls are matched by the following 
algorithm:
\begin{itemize}
\item We start by matching 
 all balls in row $L$ to a collection of balls (their partners)
in row 
$L-1$. We then match those partners in row $L-1$ to new partners in row $L-2$, 
and so on.  This determines a set of balls, each of which we label by $L$.
\item  We then take the unmatched balls in row $L-1$ and match them 
to partners in row $L-2$.  We then match those partners in row $L-2$ to 
new partners in row $L-3$, and so on.  This determines a set of balls,
each of which we label by $L-1$.
\item We continue in this way, determining a set of balls labeled 
$L-2$, $L-3$, and so on, and finally we label any unmatched balls in row
$1$ by $1$.
\item If at any point there's a free (unmatched) 
ball $b'$ directly underneath the ball $b$ we're matching,
we must match $b$ to $b'$.  We say that $b$ and $b'$ are
\emph{trivially paired}.
\end{itemize}
Let $\mu= (\mu_1,\dots,\mu_n) \in \{0,1,\dots,L\}^n$ be the labeling
of the balls in row $1$ at the end of this process (where
an empty position is denoted by $0$).  We then say that 
$Q$ is a \emph{multiline queue of type $\mu$}, and we call $\MLQ(\mu)$ the set of all multiline queues of type $\mu$.
See \cref{fig:MLQ_example} 
	for an example.
\end{definition}
\begin{figure}[!ht]
  \centerline{\includegraphics[height=1in]{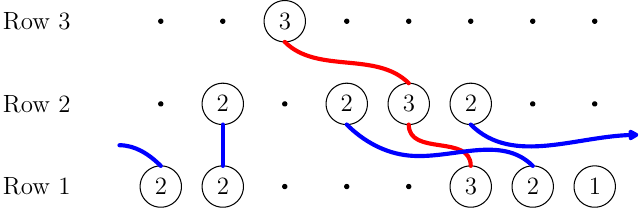}}
\centering
 \caption{
A multiline queue of type $(2,2,0,0,0, 3,2,1)$.}
\label{fig:MLQ_example}
 \end{figure}

\begin{remark}
Note that the induced labeling on the balls satisfies the following
properties:
\begin{itemize}
\item If ball $b$ with label $i$ is directly above 
ball $b'$ with label $j$, then we must have $i \leq j$.
\item Moreover if $i=j$, then those two balls are matched to each other.
\end{itemize}
\end{remark}

We now define the weight of each multiline queue. Here we generalize Martin's ideas \cite{Martin} by adding parameters
$q$ and $x_1,\ldots ,x_n$.
\begin{definition}\label{def:wt}
Given a multiline queue $Q$, we let 
$m_i$ be the number of balls in column $i$.
We define the \emph{$\mathbf x$-weight} of $Q$ to be 
$\wt_x(Q) = x_1^{m_1} x_2^{m_2} \dots x_n^{m_n}$.

We also define the \emph{$q,t$-weight} of $Q$ by associating a weight to each nontrivial pairing $p$ of balls. These weights are computed in order as follows. Consider the nontrivial pairings between rows $r$ and $r-1$.  We read the balls in row $r$ in decreasing order of their label (from $L$ to $r$); within a fixed label, we read the balls from right to left.  As we read the balls in this order, we imagine placing the strands pairing the balls one by one.  The balls that have not yet been matched are considered  \emph{free}. If pairing $p$ matches ball $b$ in row $r$ and column $c$ to ball $b'$ in row $r-1$ and column $c'$, then the free balls in row $r-1$ and columns $c+1, c+2,\dots, c'-1$  (indices considered modulo $n$) are considered \emph{skipped}. When pairing balls of label $i$ between rows $r$ and $r-1$, trivially paired balls of label $i$ in row $r-1$ are not considered free. Let $i$ be the label of balls $b$ and $b'$. We then associate to pairing $p$ the weight
\[
\wt_{q,t}(p) = \begin{cases}
\frac{(1-t) t^{\# \skipped}}{1-q^{i-r+1}t^{\# \free}}\cdot q^{i-r+1}
&\mbox {if $c'<c$}\\
\frac{(1-t) t^{\# \skipped}}{1-q^{i-r+1}t^{\# \free}}
&\mbox{if $c'>c$}.
\end{cases}
\]
Note that the extra factor $q^{i-r+1}$ appears precisely when 
the strand connecting $b$ to $b'$ wraps around the cylinder.

Having associated a $q,t$-weight to each nontrivial pairing of balls,
we define  the $q,t$-weight of the multiline queue $Q$ to be 
$$\wt_{q,t}(Q) = \prod_p \wt_{q,t}(p),$$
where the product is over all nontrivial pairings of balls in $Q$.

Finally the \emph{weight} of $Q$ is defined to be 
$$\wt(Q) = \wt_{x}(Q) \wt_{q,t}(Q).$$
\end{definition}

\begin{example}
In \cref{fig:MLQ_example}, 
the $\mathbf x$-weight of the multiline queue 
$Q$ is 
$x_1 x_2^2 x_3 x_4 x_5 x_6^2 x_7 x_8$.

The weight of the unique pairing between row $3$ and row $2$ is 
$\frac{(1-t)t}{1-qt^4}$.  The weight of the pairing of balls labeled
$3$ between row $2$ and $1$ is 
$\frac{(1-t)}{1-q^2 t^5}$, and the weights of the pairings of balls
labeled $2$ are 
$\frac{(1-t) t^2}{1-qt^3} \cdot q$ and 
$\frac{1-t}{1-qt^2}$.  
Therefore 
$$\wt(Q) = 
x_1 x_2^2 x_3 x_4 x_5 x_6^2 x_7 x_8
\cdot \frac{(1-t)t}{1-qt^4}
\cdot \frac{(1-t)}{1-q^2 t^5}
\cdot \frac{(1-t) t^2}{1-qt^3} \cdot q
\cdot \frac{1-t}{1-qt^2}.  $$
\end{example}

We now define the weight-generating function for multiline queues
of a given type, as well as the \emph{combinatorial partition function}
for multiline queues.
\begin{definition}\label{def:Fmu}
Let $\mu = (\mu_1,\dots, \mu_n) \in \{0,1,\dots, L\}^n$ be a weak composition with largest part $L$.  
We set 
\[
F_{\mu} = F_{\mu}(x_1,\dots,x_n; q, t)  = F_{\mu}({\bf x}; q, t) 
 = \sum_Q \wt(Q),\] 
 where the sum is over all $L \times n$ multiline queues of type $\mu$.

Let $\lambda = (\lambda_1 \geq \lambda_2 \geq \dots \geq \lambda_n \geq 0)$
be a partition with $n$ parts and largest part $L$.
We set 
$$Z_{\lambda} = Z_{\lambda}(x_1,\dots,x_n; q, t) = 
	Z_{\lambda}({\bf x}; q, t) = \sum_{\mu\in S_n(\lambda)} F_{\mu}(x_1,\dots,x_n; q, t).$$
	We call 
$Z_{\lambda}$ the \emph{combinatorial partition function} for multiline queues.
\end{definition}


\subsection{The main results}

The goal of this article is to show that with the refined statistics given
in \cref{def:wt}, we can use multiline queues to give formulas for 
Macdonald polynomials.  We also obtain a new proof of Martin's result that multiline queues give
steady state probabilities in the multispecies ASEP.

\begin{proposition}\label{prop:nonsymmetric}
Let $\lambda$ be a \emph{partition}.
	Then the nonsymmetric Macdonald
polynomial $E_{\lambda}(\mathbf{x}; q, t)$ is equal to 
 the quantity $F_{\lambda}(\mathbf{x}; q, t)$ 
from \cref{def:Fmu}.
\end{proposition}

\begin{theorem}
	\label{thm:main}
Let $\lambda$ be a partition.
	Then the symmetric Macdonald polynomial $P_{\lambda}(\mathbf{x}; q, t)$
	is equal to 
 the quantity  $Z_{\lambda}(\mathbf{x}; q, t)$ 
from \cref{def:Fmu}.
\end{theorem}

See \cref{fig:E} for an example illustrating \cref{prop:nonsymmetric}.
\begin{figure}[!ht]
  \centerline{\includegraphics[width=\linewidth]{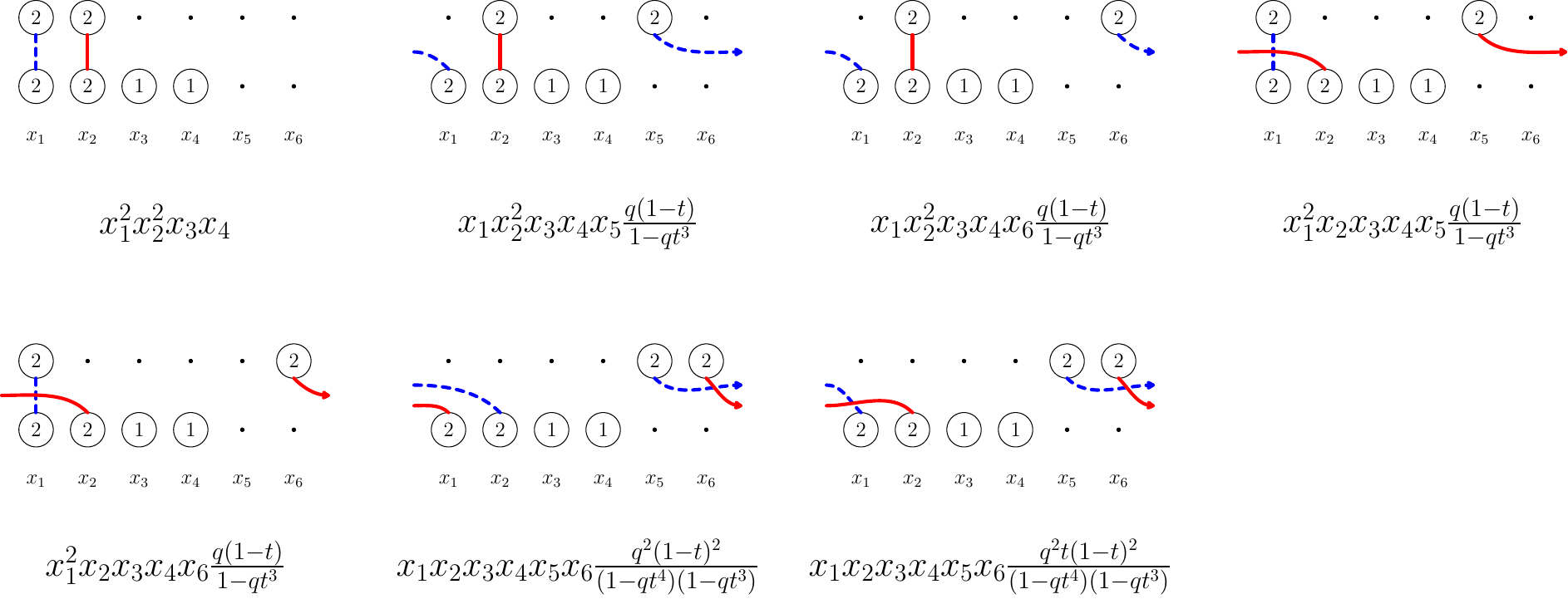}}
\centering
\caption{The generating function for the multiline queues of type 
$(2,2,1,1,0,0)$ gives an expression for the nonsymmetric Macdonald 
polynomial $E_{(2,2,1,1,0,0)}({\mathbf x}; q, t)$.}
\label{fig:E}
 \end{figure}

Although he used slightly different conventions, the following result is essentially
the same as the main result of \cite{Martin}.
\begin{cor} \label{main:prob}
	Let $\lambda$ be a partition and let $\mu$ be a composition obtained by 
	rearranging the parts of $\lambda$.
	Set $x_1 = \dots = x_n = q=1$ in $F_{\mu}$ and $Z_{\lambda}$.  Then the 
	steady state probability of being in state $\mu$ of $\ASEP(\lambda)$ is
	$\frac{F_{\mu}}{Z_{\lambda}}$.
\end{cor}

We also show in \cref{prop:PB} that 
for any composition $\mu$, 
the polynomial $F_{\mu}(\mathbf{x}; q, t)$  is equal to 
a \emph{permuted basement Macdonald polynomial}.  
Using \cref{prop:PB} 
and 
 \cref{thm:main}, we obtain
the following corollary.
\begin{cor}
	The Macdonald polynomial $P_{\lambda}(\mathbf{x}; q, t)$  can 
	be expressed as
	$$P_{\lambda}(\mathbf{x}; q, t) = \sum_{\mu \in S_n(\lambda)}
E_{\inc(\mu)}^{\sigma},$$
where 
$E_{\inc(\mu)}^{\sigma}$ is a \emph{permuted basement
Macdonald polynomial} \cite{thesis, A}, 
$\inc(\mu)$ is the sorting of the parts of $\mu$ in increasing order,
and $\sigma$ is the longest permutation such that 
$\mu_{\sigma(1)} \leq \mu_{\sigma(2)} \leq 
\dots \leq \mu_{\sigma(n)}$.
\end{cor}

\begin{remark}
It would be interesting to extend \cref{prop:nonsymmetric} to give a multiline queue formula
for all nonsymmetric Macdonald polynomials, not just those indexed by partitions. We leave this as an open problem.
\end{remark}

\begin{remark}
 The multispecies TASEP (i.e. the case $t=0$) and multiline queues have 
 been recently connected to the combinatorial $R$-matrix 
 and tensor products of KR-crystals
 \cite{KMO15, AasGrinbergS}.  Our main results are consistent
 with these results on KR-crystals, in view of the fact that Macdonald
 polynomials at $t=0$ agree with the graded characters of 
 KR-modules \cite{Lenart1, Lenart2}.
 \end{remark}


 \begin{remark}
A potentially useful probabilistic interpretation of a multiline queue when $q=1$ is as a  series of priority queues in discrete time with a Markovian service process. A single priority queue is made up of two rows, where the top row contains customers ordered by priority with the column containing each customer representing his arrival time (modulo $n$, the total number of columns). The bottom row of the queue contains services, such that the column containing a service represents the time the service occurs (modulo $n$). At his turn, a customer considers every service offered to him and declines an available service with probability $t$ and accepts with probability $1-t$ (with the exception that if the service occurs at the time of his arrival, then he accepts with probability 1). Once a service is accepted, the service is no longer available.

	 {Note that we allow a customer to decline all services, but then wrap around and consider the services again in order. Consequently, if $f$ is the number of free (available) services the customer is considering, and $0\leq s \leq f-1$, the probability of a customer accepting the next service immediately after declining $s \mod f$ services is 
\[
\sum_{n \geq 0} t^{s+nf}(1-t) = \frac{t^{s}(1-t)}{1-t^{f}}.
\] 
To match the weight of a pairing in \cref{def:wt}, set $f=\free$ and $s=\skipped$. 
}

It would be interesting to extend this 
	 interpretation to the case of generic $q$.
\end{remark}

\subsection{The Hecke algebra, ASEP, and Macdonald polynomials}

To explain the connection between the ASEP and Macdonald polynomials, and explain how we 
prove \cref{prop:nonsymmetric} and \cref{thm:main},  we need to introduce 
the Hecke algebra and recall some notions from
\cite{KasataniTakeyama} and
Cantini-deGier-Wheeler \cite{CGW}.

\begin{defn}
	The \emph{Hecke algebra} of type $A_{n-1}$ is the $\mathbb{C}$-algebra with generators $T_i$ for $1 \leq i \leq n-1$
and parameter $t$
which satisfies the following
relations:
\begin{equation}\label{Hecke}
(T_i-t)(T_i+1)=0, \qquad T_i T_{i \pm 1} T_i = T_{i \pm 1} T_i T_{i \pm 1}, \qquad
T_i T_j = T_j T_i \text{ when } |i-j|>1.
\end{equation}
\end{defn}

There is an action of the Hecke algebra on polynomials $f(x_1,\dots,x_n)$ which is defined as follows:
\begin{equation}\label{Heckeaction}
T_i=t-\frac{t x_i-x_{i+1}}{x_i-x_{i+1}}(1-s_i) 
\text{ for }1 \leq i \leq n-1,
\end{equation}

where the simple transposition $s_i$ acts on polynomials  by
\begin{equation}\label{action1}
s_i f(x_1,\ldots, x_i,x_{i+1},\ldots,x_n):=f(x_1,\ldots, x_{i+1},x_{i},\ldots,x_n).
\end{equation}  One can check that the operators 
\eqref{Heckeaction} satisfy the relations \eqref{Hecke}.

	We also define the action of the shift operator $\omega$ on polynomials via 
\begin{equation}\label{shift}
(\omega f)(x_1,\dots,x_n) = f(q x_n,x_1,\dots, x_{n-1}).
\end{equation}


Given a composition $\mu = (\mu_1,\dots,\mu_n)$, we let $|\mu|:=\sum \mu_i$.  
We also define 
\begin{align}\label{action2}
	s_i \mu &:= s_i (\mu_1,\dots,\mu_n) = (\mu_1,\dots, \mu_{i+1},\mu_i,\dots,\mu_n) \text { for }1\leq i \leq n-1, \text{ and }\\
	\omega \mu &:= \omega (\mu_1,\dots,\mu_n) = (\mu_n, \mu_1,\dots, \mu_{n-1}).\label{action3}
\end{align}

The following notion of \emph{qKZ family} was introduced in 
\cite{KasataniTakeyama}, also explaining the relationship of such 
polynomials to nonsymmetric Macdonald polynomials.
We use the conventions of 
	\cite[Definition 2]{CGW-arxiv}, see also 
 \cite[Section 1.3]{CGW}  
and \cite[(23)]{CGW}.

\begin{defn}\label{f-def} 
Fix a partition $\lambda = (\lambda_1,\dots,\lambda_n)$.
	We say that a family 
	$\{f_{\mu} \}_{\mu 
	\in S_n(\lambda)}$ 
	of homogeneous 
	degree $|\lambda|$ 
	polynomials in $n$ variables 
	$\mathbf{x} = (x_1,\dots,x_n)$, with coefficients which are rational functions of 
	$q$ and $t$, is a \emph{qKZ family} if they satisfy 
	\begin{align}
T_i f_{\mu}(\mathbf{x}; q,t) &= f_{s_i \mu} (\mathbf{x};q,t),
\text{ when }\mu_i > \mu_{i+1}, \label{firstproperty}\\
	T_i f_{\mu}(\mathbf{x}; q,t) &= t f_{\mu}(\mathbf{x}; q,t), \text{ when }\mu_i = \mu_{i+1}, \label{secondproperty}\\
q^{\mu_n} f_{\mu}(\mathbf{x}; q,t)
&= 
		f_{\mu_n,\mu_1,\dots,\mu_{n-1}}(q x_n,x_1,\dots,x_{n-1}; q,t).  \label{thirdproperty}
\end{align}
\end{defn}

\begin{remark}
Note that \eqref{thirdproperty} can be rephrased as 
\[
q^{\mu_n} f_{\mu}(\mathbf{x}; q,t)
= 
		(\omega f_{\mu_n,\mu_1,\dots,\mu_{n-1}})(\mathbf{x}; q,t).
		\]
\end{remark}

The following lemma explains the relationship of the $f_{\mu}$'s to the ASEP.
\begin{lemma}\label{ASEP-f}
	\cite[Corollary 1]{CGW-arxiv}.
	Consider the polynomials $f_{\mu}$ from \cref{f-def}.
	When $q=x_1=\dots=x_n=1$, 
 $f_{\mu}(1,\dots,1; 1, t)$ is proportional to the steady 
state probability that the multispecies ASEP is in state 
$\mu$.
\end{lemma}


As we will explain in 
\cref{Elemma} and 
	\cref{lem:Macdonald}, 
the polynomials $f_{\mu}$ are also related to  Macdonald polynomials.
We first quickly review the relevant definitions.

\begin{definition}\label{def:Macdonald}
Let $\langle \cdot, \cdot \rangle$ denote the 
Macdonald inner product on power sum symmetric functions 
\cite[Chapter VI, (1.5)]{Macdonald}, where $<$ denotes
the dominance order on partitions.  
Let $\lambda$ be a partition.
	The (symmetric) \emph{Macdonald polynomial}
$P_{\lambda}(x_1,\dots,x_n; q, t)$ is the unique homogeneous
symmetric polynomial in $x_1,\dots, x_n$ which satisfies
\begin{align*}
\langle P_{\lambda}, P_{\mu}\rangle &=0, \ \lambda \neq \mu,\\
P_{\lambda}(x_1,\dots,x_n; q, t) &= m_{\lambda}(x_1,\dots,x_n)+
\sum_{\mu< \lambda} c_{\lambda, \mu}(q,t) m_{\mu}(x_1,\dots,x_n),
\end{align*}
	i.e. the coefficients $c_{\lambda, \mu}(q,t)$ 
	 are completely determined by the orthogonality conditions.
\end{definition}

The following definition can be found in 
\cite{MacdonaldBourbaki} (see also \cite{Marshall} for a nice 
exposition).
\begin{definition}\label{def:nonsymmetric}
For $1 \leq i \leq n$, we define the \emph{$q$-Dunkl} or \emph{Cherednik 
operators} \cite{Cherednik1, Cherednik2} by 
$$Y_i = T_i^{-1} \dots T_{n-1}^{-1} \omega T_1 \dots T_{i-1}.$$

The Cherednik operators commute pairwise, and hence possess a set of 
simultaneous eigenfunctions,  which are (up to scalar)
	the \emph{nonsymmetric Macdonald polynomials}.  Each simultaneous 
	eigenspace is one-dimensional.  We index
	the nonsymmetric Macdonald polynomials $E_{\mu}(\mathbf{x}; q, t)$ by 
compositions $\mu$ so that 
	$$E_{\mu}(\mathbf{x}; q, t) = \mathbf{x}^{\mu} + \sum_{\nu < \mu} 
	b_{\mu \nu}(q,t) \mathbf{x}^{\nu},$$
	where the partial order on compositions is as in \cite[(2.15)]{Marshall}.

	There is an explicit formula for each eigenvalue of $Y_i$ acting
	on the nonsymmetric Macdonald polynomial $E_{\mu}$ \cite[(2.13)]{Marshall};
	in particular, when $\lambda = (\lambda_1 \geq  \dots
	\geq \lambda_n \geq 0)$ is a partition,  
 we have that for $1 \leq i \leq n$, 
	\begin{equation}\label{eigenfunction}
Y_i E_{\lambda} = y_i(\lambda) E_{\lambda}
\end{equation}
where 
\[
y_i(\lambda) =
q^{\lambda_i}t^{\#\{j<i| \lambda_j=\lambda_i\}- \#\{j>i| \lambda_j=\lambda_i\}}.
\]
\end{definition}

\cref{Elemma} below essentially appears in \cite[Section 3.3]{KasataniTakeyama}.
We thank Michael Wheeler for his explanations.

\begin{lemma} \label{Elemma}
Let $\lambda = (\lambda_1,\dots,\lambda_n)$ be a partition
and let $\{f_{\mu} \}_{\mu
	\in S_n(\lambda)}$ be a set of homogeneous degree $|\lambda|$ polynomials 
as in \cref{f-def}.  
Then $f_{\lambda}$ is a scalar multiple of the nonsymmetric Macdonald polynomial
$E_{\lambda}$.
\end{lemma}

\begin{proof}
For 
	$1 \leq i \leq n$, we claim that \eqref{eigenfunction}
	holds with $E_{\lambda}$ replaced by $f_{\lambda}$, i.e. 
$$Y_i f_{\lambda} = y_i(\lambda) f_{\lambda}.$$

This is because acting by $T_{i-1}$, followed by $T_{i-2}$, and so on, up to $T_1$, means we apply \eqref{firstproperty} when $\lambda_j>\lambda_i$ and \eqref{secondproperty} when $\lambda_j=\lambda_i$ for $j<i$, where the latter contributes a factor of $t$. Thus 
\[Y_i f_{\lambda}= t^{\#\{j<i| \lambda_j=\lambda_i\}}T_i^{-1} \dots T_{n-1}^{-1} \omega f_{(\lambda_i, \lambda_1,\ldots,\lambda_{i-1},\lambda_{i+1},\ldots,\lambda_n)}.
\]
Acting by $\omega$ on $f_{(\lambda_i, \lambda_1,\ldots,\lambda_{i-1},\lambda_{i+1},\ldots,\lambda_n)}$ gives $q^{\lambda_i}f_{(\lambda_1,\ldots,\lambda_{i-1},\lambda_{i+1},\ldots,\lambda_n,\lambda_i)}$. Finally, by \eqref{firstproperty}, $T_j^{-1}f_{\mu}=f_{s_j\mu}$ when $\mu_j<\mu_{j+1}$, from which we obtain the desired equality by applying $T^{-1}_{n-1},\ldots,T^{-1}_i$ in that order. 	

	Therefore by \cref{def:nonsymmetric}, 
	$f_{\lambda}$ must be a scalar multiple of 
	$E_{\lambda}$.
\end{proof}


\begin{lemma}\label{lem:Macdonald}\cite[Lemma 1]{CGW-arxiv}
	Let $\lambda$ be a partition.  Then the  Macdonald polynomial
	$P_{\lambda}(x_1,\dots,x_n;q,t)$ is a scalar multiple of 
	$$\sum_{\mu\in S_n(\lambda)} f_{\mu}(x_1,\dots,x_n;q,t).$$

\end{lemma}

\begin{proof}
	The symmetric Macdonald polynomial $P_{\lambda}$
	is the unique polynomial in the subspace
	$V_{\lambda}:= \Q(q,t) \{E_{\mu} \ \vert \ 
	\mu \in S_n(\lambda) \}$ which is invariant under 
	$S_n$ and such that the coefficient of ${\mathbf x}^{\lambda}$
	is $1$ \cite[Section 5.3]{MacdonaldAffine}, see also 
	\cite[Section 6.18]{HaimanICM}.

	It follows from  \cref{Elemma}, the definition of the $f_{\mu}$ and the 
	fact that $V_{\lambda}$ is a module for the Hecke algebra
	\cite[Section 6.18]{HaimanICM} that 
	$\sum_{\mu} f_{\mu}$ lies in $V_{\lambda}$.  

	Finally it is straightforward to show that if $\mu_i > \mu_{i+1}$,
	then $T_i (f_{\mu} + f_{s_i \mu}) = t
	(f_{\mu}+f_{s_i \mu})$, which together with 
	\eqref{secondproperty}, shows that 
	$T_i \sum_{\mu} f_{\mu} = t \sum_{\mu} f_{\mu}$. This is 
	equivalent to the fact that $\sum_{\mu} f_{\mu}$ is symmetric
	in $x_i$ and $x_{i+1}$, and hence $\sum_{\mu} f_{\mu}$
	is invariant under $S_n$.
\end{proof}

The strategy of our proof of \cref{thm:main} is very simple.  
Our main task is to show that the $F_{\mu}$'s 
satisfy the following properties.

\begin{theorem}\label{thm:123}
\begin{align}
\label{a} T_i F_{\mu}(\mathbf{x}; q,t) &= F_{s_i \mu} (\mathbf{x};q,t),
\text{ when }\mu_i > \mu_{i+1},\\
\label{b} T_i F_{\mu}(\mathbf{x}; q,t) &= t F_{\mu}(\mathbf{x}; q,t), \text{ when }\mu_i = \mu_{i+1},\\
q^{\mu_n} F_{\mu}(\mathbf{x}; q,t)
&= 
\label{c} F_{\mu_n,\mu_1,\dots,\mu_{n-1}}(q x_n,x_1,\dots,x_{n-1}; q,t). 
\end{align}
\end{theorem}

Once we have done this, we verify the following lemma.
\begin{lemma}\label{monic}
For any partition $\lambda$, 
$$F_{\lambda}(\mathbf{x};q,t) = E_{\lambda}(\mathbf{x}; q, t),$$ where $E_{\lambda}$ is the nonsymmetric Macdonald polynomial.
\end{lemma}

\begin{proof}
	By \cref{Elemma}, we know that $F_{\lambda}$ is a scalar multiple of $E_{\lambda}$.
It follows from the definition that the
 coefficient of ${\mathbf x}^{\lambda}$ in 
$F_{\lambda}$ is $1$, and it  follows from \cref{def:nonsymmetric} that  
	 the coefficient of ${\mathbf x}^{\lambda}$ in  $E_{\lambda}$ is $1$,
	 so we are done.
\end{proof}

Then  \cref{thm:123},   \cref{monic}, 
and \cref{lem:Macdonald} implies \cref{thm:main}, 
that our sum over multiline queues equals the symmetric Macdonald 
polynomial $P_{\lambda}$.  

Note also that once we have verified  \cref{thm:123},   \cref{ASEP-f} implies 
\cref{main:prob}, the formula for probabilities of $\ASEP(\lambda)$ in terms of 
multiline queues.

\begin{remark}
It is straightforward to check, using the definition of the action of the 
$T_i$'s in \eqref{Heckeaction}, that 
\eqref{a} is equivalent to the statement that if $\mu_i>\mu_{i+1}$,
\begin{equation}\label{aa}
\frac{(1-t)x_{i+1}}{x_i-x_{i+1}} F_{\mu}(\mathbf{x}; q, t) + 
\frac{(tx_i-x_{i+1})}{x_i-x_{i+1}} s_i F_{\mu}(\mathbf{x}; q, t) - 
 F_{s_i \mu}(\mathbf{x}; q, t) = 0.
\end{equation}

Similarly, 
\eqref{b} is equivalent to the statement that if $\mu_i=\mu_{i+1}$,
\begin{equation}\label{bb}
F_{\mu}(\mathbf{x}; q, t) = s_i  F_{\mu}(\mathbf{x}; q, t).
\end{equation}
In other words, when $\mu_i=\mu_{i+1}$, 
$F_{\mu}(\mathbf{x}; q, t)$ is symmetric in $x_i$ and $x_{i+1}$.
\end{remark}

The structure of this paper is as follows.  
In \cref{circular}, we  prove that the $F_{\mu}$'s satisfy 
\eqref{c}, the circular symmetry, 
and in \cref{inductive}, we  use induction to prove that all multiline queue
generating functions
satisfy \eqref{aa} and \eqref{bb}.
This completes the proof of our main results.
In \cref{sec:comparison} we show that our polynomials $F_{\mu}$ agree
with certain \emph{permuted basement Macdonald polynomials}, and we 
compare the number of terms in our formula versus the Haglund-Haiman-Loehr
formula for $E_{\lambda}$.
In \cref{sec:tableau} we give a bijection between multiline queues and
some tableaux we call \emph{queue tableaux}; the latter coincide with 
permuted
basement tableaux precisely when $\mu$ is a composition with all parts distinct.


\vskip .2cm

\noindent{\bf Acknowledgements:~} We would like to thank 
James Martin, for sharing an early draft of his paper \cite{Martin} 
with us.  We would also like to thank Mark Haiman for several
interesting conversations about Macdonald polynomials, and Jim Haglund
for telling us about permuted basement Macdonald polynomials.
We are grateful  to Jan de Gier and Michael Wheeler
for useful explanations of their results \cite{CGW, CGW-arxiv}, and 
Sarah Mason for helpful comments on our paper.  Finally, we would like to 
thank the referee for a remarkable referee report, whose 97 
 comments/suggestions were extremely helpful.

\section{Circular symmetry: the proof of \eqref{c}} \label{circular}
In this section we prove \eqref{c}, which we restate in \cref{prop:circular}
for convenience.

We start by proving a useful lemma regarding the weight of pairings in multiline queues. Recall from \cref{def:wt} that $\wt_{q,t}(p)$ is defined using the particular pairing order in which within each row, balls with the same label are paired from right to left. The following lemma says that any other pairing order would give rise to the same weight-generating function for multiline queues.

\begin{lemma}\label{lem:order}
Let $Q$ be a multiline queue, and let $r,r-1$ be two consecutive rows. Let $k\geq r$. Then the total weight of all possible pairings of the balls with label $k$ from row $r$ to the balls with label $k$ in row $r-1$ is independent of the order in which we pair those balls. 

	That is, suppose that 
 there are $\ell$ balls with label $k$ in row $r$ which are nontrivially
paired to partners in row $r-1$.  We denote the $\ell$ balls in row
$r$ by $b_1,\dots,b_{\ell}$ from right to left.  
Now given a permutation $\pi\in S_{\ell}$, 
let us modify \cref{def:wt} by reading the balls $\{b_1,\dots,b_{\ell}\}$ in the order $b_{\pi(1)},\dots, b_{\pi(\ell)}$, and let 
$\wt^{\pi}_{q,t}(p_{\pi(1)}),\dots, \wt^{\pi}_{q,t}(p_{\pi(\ell)})$
denote the resulting weights of the pairings, where 
$p_{\pi(i)}$ denotes the pairing of $b_{\pi(i)}$ to its partner.
Then for any two permutations
	 $\pi,\pi'\in S_{\ell}$, we have the following:
\[
	\sum_{p_1,\ldots,p_\ell} \prod_{i=1}^\ell \wt_{q,t}^{\pi}(p_{\pi(i)}) = \sum_{p_1,\ldots,p_\ell} \prod_{i=1}^\ell \wt_{q,t}^{\pi'}(p_{\pi'(i)}),
\]
where the sums are over all possible pairings $p_1,\ldots,p_\ell$ between 
the balls with label $k$ in row $r$ and row $r-1$. 
\end{lemma}

\begin{proof}
It is enough to show this holds when $\pi'=s_j\pi$ for some transposition $s_j$. Consider the balls $b_{\pi(j)}$ and $b_{\pi(j+1)}$ in row $r$, and suppose a given pairing matches them to balls in row $r-1$ which we denote by $b'_{\pi(j)}$ and $b'_{\pi(j+1)}$ respectively. We construct an involution $\iota$ on pairings as follows. If the pairings $b_{\pi(j)}  \rightarrow b'_{\pi(j)}$ and $b_{\pi(j+1)} \rightarrow b'_{\pi(j+1)}$ both cross each other (meaning that pairing $b_{\pi(j)}  \rightarrow b'_{\pi(j)}$  skips over ball $b'_{\pi(j+1)}$, and pairing $b_{\pi(j+1)} \rightarrow b'_{\pi(j+1)}$ skips over ball $b'_{\pi(j)}$) $\iota$ does nothing. If neither of these pairings cross each other, $\iota$ does nothing. Otherwise, $\iota$ swaps the two endpoints of the pairings, so that $\iota(p_{\pi(j)})$ is a pairing from $b_{\pi(j)}$ to $b'_{\pi(j+1)}$ and $\iota(p_{\pi(j+1)})$ is a pairing from $b_{\pi(j+1)}$ to $b'_{\pi(j)}$.
	
We claim that 
\begin{equation}\label{eq:iota}
\wt_{q,t}^{\pi}(p_{\pi(j)})\wt_{q,t}^{\pi}(p_{\pi(j+1)}) = \wt_{q,t}^{\pi'}(\iota(p_{\pi'(j)}))\wt_{q,t}^{\pi'}(\iota(p_{\pi'(j+1)})).
\end{equation} 
This is not hard to check. When $\iota$ acts trivially, $\wt_{q,t}^{\pi}(p_{\pi(i)})=\wt_{q,t}^{\pi'}(\iota(p_{\pi'(i)}))$ for $i=j,j+1$. We will illustrate it in the case that $\iota$ swaps the endpoints of the pairings. Without loss of generality, suppose the pairing $p_{\pi(j)}$ from $b_{\pi(j)}$ to $b'_{\pi(j)}$ skips over the ball $b'_{\pi(j+1)}$. Suppose $m_1$ of the balls from the set $S=\{b'_{\pi(j+2)},\ldots,b'_{\pi(\ell)}\}$ lie between $b_{\pi(j)}$ and $b'_{\pi(j+1)}$,  $m_2$ of the balls from $S$ lie between $b_{\pi(j+1)}$ and $b'_{\pi(j+1)}$, and $m_3$ of the balls from $S$ lie between $b'_{\pi(j+1)}$ and $b'_{\pi(j)}$, as in the figure below. 
\begin{figure}[!ht]
  \centerline{\includegraphics[width=\linewidth]{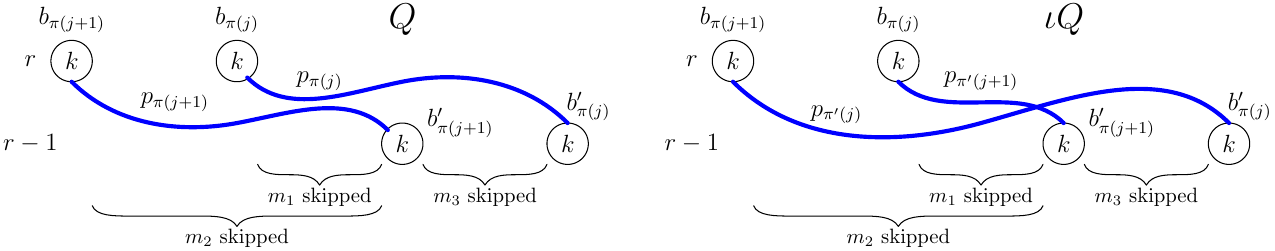}}
\end{figure}
Note that $b'_{\pi(j+1)}$ lies between $b_{\pi(j+1)}$ and $b'_{\pi(j)}$ (cyclically), since otherwise both pairings would cross each other. Then $\wt_{q,t}^{\pi'}(\iota(p_{\pi'(j)}))=\wt_{q,t}^{\pi'}(\iota(p_{\pi(j+1)})) = \wt_{q,t}^{\pi}(p_{\pi(j)}))t^{m_2-m_1}$ and $\wt_{q,t}^{\pi'}(\iota(p_{\pi'(j+1)}))=\wt_{q,t}^{\pi'}(\iota(p_{\pi(j)})) = \wt_{q,t}^{\pi}(p_{\pi(j+1)}))t^{m_1-m_2}$, which gives us \eqref{eq:iota}. See \cref{ex:iota}.
	
We therefore conclude that when we sum over all possible pairings, the total weight is the same with pairing orders $\pi$ and $\pi'=s_j\pi$, and hence this also holds for arbitrary pairing orders. 
\end{proof}

\begin{example}\label{ex:iota}
Let $Q$ be a multiline queue whose rows $r,r-1$ are shown below, with nontrivially paired balls $b_1,b_2,b_3$ with label $k$ in row $r$. Suppose $\pi=(2,1,3)$ and $\pi'=s_1\pi=(1,2,3)$. We show $\iota Q$ for this $\pi'$. (The example only includes balls having label $k$ since balls of other labels don't interact with the label $k$ pairing order.) 
\begin{figure}[!ht]
  \centerline{\includegraphics[width=\linewidth]{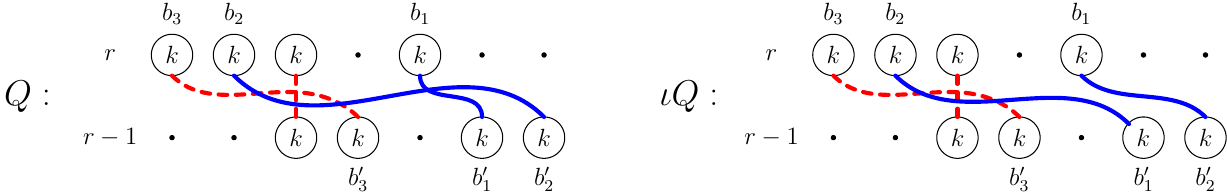}}
\end{figure}

We compute all the relevant quantities, noting that the denominators of both sides of \eqref{eq:iota} are equal.
\begin{itemize}
\item $p_{\pi(1)}$ is the pairing from $b_2$ to $b'_2$ and the numerator of $\wt_{q,t}^{\pi}(p_{\pi(1)})$ is $t^2(1-t)$.
\item $p_{\pi(2)}$ is the pairing from $b_1$ to $b'_1$ and the numerator of  $\wt_{q,t}^{\pi}(p_{\pi(2)})$ is $1-t$.
\item $\iota(p_{\pi'(1)})$ is the pairing from $b_1$ to $b'_2$ and the numerator of $\wt_{q,t}^{\pi'}(\iota(p_{\pi'(1)}))$ is $t(1-t)$.
\item $\iota(p_{\pi'(2)})$ is the pairing from $b_2$ to $b'_1$ and the numerator of  $\wt_{q,t}^{\pi'}(\iota(p_{\pi'(2)}))$ is $t(1-t)$.
\end{itemize}
We see that \eqref{eq:iota} is satisfied for $\pi,\pi',j=1$.
\end{example}

	\begin{proposition}\label{prop:circular}
\begin{equation}\label{eq:q}
F_{\mu_n,\mu_1,\ldots ,\mu_{n-1}}(qx_n, x_1,\ldots ,x_{n-1};q, t)=
q^{\mu_n}F_{\mu_1,\ldots ,\mu_{n}}(x_1,\ldots ,x_{n};q, t).
\end{equation}
\end{proposition}
Let $L = \max\{\mu_1,\dots,\mu_n\}$. Both sides of \eqref{eq:q} have an interpretation in terms of multiline queues with $L$ rows. In our proof, we will take advantage of \cref{lem:order} and use a different pairing order for the multiline queues on the RHS versus the LHS.

We define a \emph{sequence of ball labels} for any given column of a multiline queue to be the word obtained by reading the labels off the balls in that column from bottom to top, and recording a 0 for each empty spot. Let this word have the form $i_1^{k_1}\ldots i_{\ell}^{k_\ell}$ with $0\le i_j\le L$ and $k_j>0$ for any $j$. In \cref{hecke_c}, we show such a sequence of ball labels for the rightmost column of a multiline queue.

Let $\delta$ be the Kronecker delta, i.e. $\delta_{S}$ equals $1$ or $0$ based on whether $S$ is a true statement. We will prove \eqref{eq:q} by proving
the following combinatorial statement.

\begin{proposition}\label{prop:bijection}
Let $Q \in \MLQ(\mu)$, and let $\omega$ be the bijection from multiline queues to multiline queues which maps $Q$ to the cyclic shift $Q'\in\MLQ(\omega\mu)$ of $Q$, obtained by taking the $n$th column of $Q$ and wrapping it around to become the first column of $Q'$, see \cref{hecke_c} (all connectivities of balls are preserved). Let $\widetilde{\pi}$ be the pairing order for $Q'$ that pairs balls in the same order that they are paired in $Q$, regardless of their location in $Q'$.

	Then we have 
	\begin{eqnarray}
\wt_{x_1,\ldots ,x_n}(Q)&=&\wt^{\widetilde{\pi}}_{x_n,x_1,\ldots ,x_{n-1}}(Q')\label{eq:xweight}\\
q^{\mu_n}\wt_{q,t}(Q)&=&\wt^{\widetilde{\pi}}_{q,t}(Q')\prod_{j=1}^\ell
		q^{\delta_{(i_j>0)}k_j}. \label{eq:lastcolumn}
\end{eqnarray}
\end {proposition}

\begin{proof}
We first note \eqref{eq:xweight} is immediate, and moreover that the cyclic shift of the multiline queue doesn't affect any of the pairings between the balls, so by \cref{lem:order} the $t$-weight is unchanged: $\wt_{q,t}(Q)\vert_{q=1}=\wt^{\widetilde{\pi}}_{q,t}(Q')\vert_{q=1}$. Furthermore, the denominators of $Q$ and $Q'$ are identical, since they depend solely on the set of trivial pairings, and those are also preserved under the cyclic shift.  Thus it is sufficient to show equality in the $q$-weights of the numerators of both sides of \eqref{eq:lastcolumn}. 

We start by computing the weight in $q$ of the numerator of $Q$. The sequence of ball labels in the $n$th column of $Q$ is $i_1^{k_1}\ldots i_{\ell}^{k_\ell}$ with $0\le i_j\le L$, $i_j\neq i_{j+1}$, and $k_j>0$ for any $j$, as in the left side of \cref{hecke_c}. Note that $\mu_n=i_1$.

\begin{figure}[!ht]
  \centerline{\includegraphics[height=3.5in]{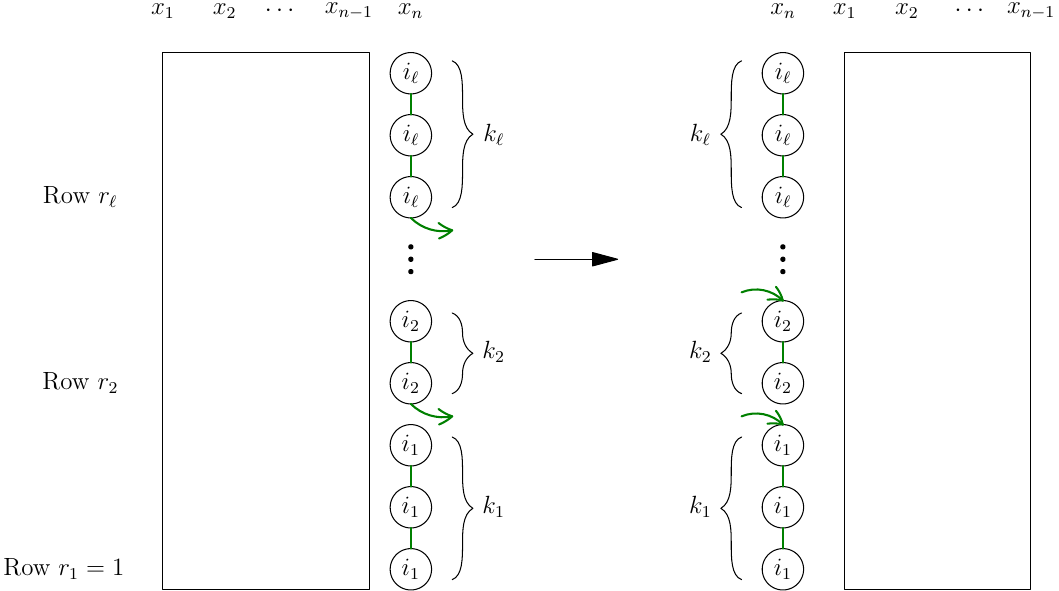}}
\centering
	\caption{The bijection $\omega$ taking 
a multiline queue $Q$ of type $(\mu_1,\ldots,\mu_n)$ (left) to its cyclic shift $Q'$ of type $(\mu_n,\mu_1,\ldots,\mu_{n-1})$ (right). 
	The column that got wrapped around has the sequence of ball labels $i_1^{k_1}\ldots i_{\ell}^{k_\ell}$. 
	In $Q$ on the left, the arrow from the lowest ball labeled $i_j$ represents a wrapping ball pairing if $i_j>0$, contributing $q^{\delta_{(i_j>0)}( i_j-r_j+1)}$ to the total weight. In $Q'$ on the right, whenever $i_j>0$, there must be an arrow going to the highest ball labeled $i_j$ (which is in row $r_{j+1}-1$) from some ball labeled $i_j$ in row $r_{j+1}$. This pairing contributes $q^{\delta_{(i_j>0)}(i_j - r_{j+1}+1)}$ to the total weight.}
	\label{hecke_c}
 \end{figure}

We also note that any ball pairing that wraps in $Q$ from a column other than the $n$'th one, will also wrap in $Q'$, so its contribution to the weight in $q$ of the numerator is identical on both sides of \eqref{eq:lastcolumn}. Thus let us compute the contribution to the $q$-weight in the numerator arising from pairings to or from balls in the $n$'th column of $Q$, and compare this to the $q$-weight in the numerator arising from the pairings to or from balls in the first column of $Q'$, which that column is sent to after the cyclic shift. 

A ball labeled  $i$ in column $n$ and row $r$ contributes a $1$ if there is a ball with the same label directly beneath it, and otherwise contributes $q^{i-r+1}$ to the $q$-weight of $Q$, since its pairing necessarily wraps. For $1 \leq j \leq \ell$, define $r_j = 1+\sum_{u< j} k_u$ to be the row number of the bottom-most ball labeled $i_j$ in its block. For any $j=2, \dots, \ell$ and $i_j>0$, the weight of the ball pairing wrapping from row $r_j$ is therefore 
	\[q^{i_j - r_j+1}.\]
Thus we get that the $n$th column contributes 
\begin{equation}\label{eq:qs}
	\prod_{j=2}^{\ell} q^{\delta_{(i_j>0)}(i_j-r_j+1)}
\end{equation}
to the $q$-weight of $Q$ (note that the sum starts with $j=2$ since the pairing from the ball $i_1$ in row $r_1=1$ does not wrap). Using the fact that $r_1=1$ and $i_1-r_1+1 = i_1 =\mu_n$, multiplying \eqref{eq:qs} by $q^{\mu_n}$ equates to rewriting the product:
\begin{equation}\label{eq:qs2}
	q^{\mu_n}\prod_{j=2}^{\ell} q^{\delta_{(i_j>0)}(i_j-r_j+1)}=\prod_{j=1}^{\ell} q^{\delta_{(i_j>0)}(i_j-r_j+1)}.
\end{equation}


For the first column of $Q'$, the sequence of balls read from bottom to top of the multiline queue is (again) $i_1^{k_1}\ldots i_{\ell}^{k_\ell}$ with $0\le i_j\le L$ and $k_j>0$ for any $j$, as shown on the right side of \cref{hecke_c}. As before, in $Q'$, all wrapping pairings to balls in columns other than the first one were also wrapping in $Q$. 
Thus let us compute the power of $q$ coming from the pairings that wrap to the balls in the first column of $Q'$.

The ball labeled $i$ in column $1$ and row $r-1$ contributes $1$ if the ball directly above it has the same label $i$, and $q^{i-r+1}$ otherwise, due to the incoming wrapping pairing from ball labeled $i$ in row $r$. Note that if $i=r-1$, the ball numbered $i$ in row $r-1$ is necessarily the topmost ball in its string with no ball in row $r$ connecting to it, and so there's no contribution from an incoming pairing; accordingly, $i-r+1=0$ in that case. Thus for any $j=1, \dots, \ell-1$ if $i_j> 0$, the $q$-weight of the wrapping pairing going to the topmost ball labeled $i_j$ (which is in row $r_{j+1}-1$) is 
\[
	q^{i_j - r_{j+1}+1}.
\]
(We exclude the $j=\ell$ case, since the topmost ball labeled $i_{\ell}$ is in the topmost row of the multiline queue and by definition has no pairings going into it.) Therefore, we get that the contribution to the $q$-weight to the right hand side of \eqref{eq:lastcolumn} coming from the first column of $Q'$ is
\[
	\prod_{j=1}^{\ell-1} q^{\delta_{(i_j>0)}(i_j - r_{j+1}+1)}.
\]
Now we use the fact that $r_{j}+k_j=r_{j+1}$ and $i_{\ell}=\delta_{(i_{\ell}>0)}(r_{\ell}+k_{\ell}-1)$ (where $i_{\ell}$ is necessarily either $0$ or $L$) to get:
\[
	q^{\sum_{j=1}^\ell \delta_{(i_j>0)}k_j}\prod_{j=1}^{\ell-1} q^{\delta_{(i_j>0)}(i_j - r_{j+1}+1)} = \prod_{j=1}^{\ell} q^{\delta_{(i_j>0)}(i_j-r_j+1)},
\]
which equals \eqref{eq:qs2}. Since we were comparing the difference in the $q$-weights in the numerators arising from wrapping pairings associated to the $n$th column of $Q$ vs. the first column of $Q'$, this proves the equality \eqref{eq:lastcolumn}.
\end{proof}

The proof of \eqref{eq:q} now follows from \cref{prop:bijection} because 
\[
\sum_{Q} q^{\mu_n}\wt_{q,t}(Q)\wt_{x}(Q)=q^{\mu_n}F_{\mu}(x_1,\ldots ,x_n;q,t)
\]
and for any pairing order $\widetilde{\pi}$,
\[
\sum_{Q'} \wt^{\widetilde{\pi}}_{q,t}(Q')\wt^{\widetilde{\pi}}_{x_n,x_1,\ldots ,x_{n-1}}(Q')q^{\sum_{j=1}^\ell \delta_{(i_j>0)}k_j}=F_{(\mu_n,\mu_1,\ldots ,\mu_{n-1})}(qx_n,x_1,\ldots ,x_{n-1};q,t).
\]

\section{The Hecke operators and multiline queues: the proof of \eqref{aa} and \eqref{bb}}\label{inductive}

Recall 
from \eqref{action1} and \eqref{action2}
that we use the notation
\begin{eqnarray*}
F_{s_i\mu}({\bf x}; q, t)
&=&
F_{\mu_1,\ldots ,\mu_{i-1},\mu_{i+1},\mu_i,\mu_{i+2},\ldots ,\mu_n}(x_1,\ldots ,x_{n};q,t)\\
s_iF_{\mu}({\bf x};q, t)
&=&
F_{\mu}(x_1,\ldots,x_{i-1},x_{i+1},x_i,x_{i+2},\ldots,x_n;q, t)
\end{eqnarray*}

For conciseness we will sometimes omit the dependence on $q$ and $t$, even $\x$, 
writing 
$F_{\mu}$ or $F_{\mu}(\bf{x})$ as an abbreviation for 
$F_{\mu}({\bf x};q, t)
=
F_{\mu_1,\ldots ,\mu_{n}}(x_1,\ldots ,x_{n};q, t).$

We give an inductive proof of the main result which is based on the fact that, we can view a 
multiline queue $Q$ with $L$ rows as a multiline queue $Q'$ with $L-1$ rows (the restriction of $Q$ to rows 
$2$ through $L$) sitting on top of 
a (generalized) multiline queue $Q_0$ with $2$ rows (the restriction of $Q$ to rows $1$ and $2$).  
Since $Q'$ occupies rows 
$2$ through $L$ and has balls labeled $2$ through $L$, we identify $Q'$ with a 
multiline queue obtained by decreasing the row labels and ball labels in the top $L-1$ rows of $Q$ 
by $1$, see \cref{fig:decomp}.  
\begin{figure}[!ht]
  \centerline{\includegraphics[width=\linewidth]{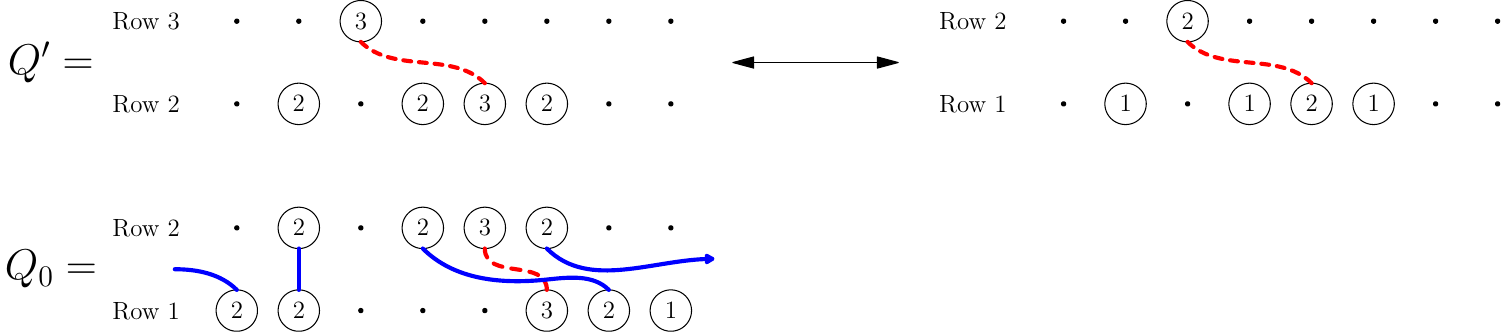}}
\centering
\caption{The multiline queue $Q$ from \cref{fig:MLQ_example} decomposes into the multiline queue $Q'$ and the 
generalized multiline queue $Q_0$ shown here.}
\label{fig:decomp}
 \end{figure}
(Holes, represented by $0$,
remain holes.)  
If the bottom row of $Q'$ is the composition $\lambda$, then after decreasing labels as above,
the new bottom row is $\lambda^-$, 
where 
$\lambda^-_i=\max(\lambda_i-1,0)$. 
Meanwhile $Q_0$ has just two rows, but its balls are labeled $1$ through $L$;
we refer to it as a \emph{generalized two-line queue}.  

\begin{definition}\label{def:two-row}
A generalized two-line queue is a two-row multiline queue whose top and bottom rows are represented by a pair of compositions $\lambda$ and $\mu$, respectively, satisfying the following conditions: $\lambda$ has no parts of size 1, and for each $j>1$, $|\{i:\mu_i=j\}|=|\{i:\lambda_i=j\}|$. Moreover, for each $i$, either $\mu_i=0$, or $\lambda_i \leq \mu_i$. (In other words, a larger label cannot be directly above a smaller nonzero label, as in a usual multiline queue, and if this condition is not satisfied, the multiline queue is not considered valid.) Let $\mathcal{Q}_{\mu}^{\lambda}$ denote the set of (generalized)
two-line queues with bottom row $\mu$ and top row $\lambda$.  
For $Q_0\in \mathcal{Q}_{\mu}^{\lambda}$, we define
\[
\wt(Q_0) = \wt_{q,t}(Q_0) \cdot \prod_{\mu_i > 0} x_i.
\]
and 
\[
F_{\mu}^{\lambda} = F_{\mu}^{\lambda}(\x) = \sum_{Q_0 \in \mathcal{Q}_{\mu}^{\lambda}} \wt(Q_0),
\]
\end{definition}

For example the queue $Q_0$ at the bottom of Figure \ref{fig:decomp} is a generalized two-line queue in $ \mathcal{Q}_{\mu}^{\lambda}$ with $\mu=(2,2,0,0,0,3,2,1)$ and $\lambda=(0,2,0,2,3,2,0,0)$.

Note that we only take the bottom row of $Q_0$ into account when computing the $\mathbf x$-weight.  This 
is because we want $\wt(Q) = \wt(Q') \wt(Q_0)$, where the top $L-1$ rows of $Q$ give $Q'$ and the 
bottom two rows give $Q_0$.

The following lemma is immediate from the definitions.

\begin{lemma}\label{lem:recursive}
\[
F_{\mu}=\sum_\lambda F_{\mu}^{\lambda}F_{\lambda^-}.
\]
\end{lemma}

\begin{remark}
Note that in \cref{lem:recursive}, since 
	$F_{\mu}^{\lambda}$ is only nonzero when
	$\lambda_i \in \{0,2,3,4,\dots\}$, we have that 
if $\lambda_i > \lambda_{i+1}$,
then $\lambda_i^- > \lambda_{i+1}^-$.
	Also note that $(s_i \lambda)^- = s_i(\lambda^-)$ so we 
	can write $s_i \lambda^-$ without any ambiguity.
\end{remark}







%

In this section we will prove \eqref{aa} and \eqref{bb}.  
Actually we will  prove  a result which implies \eqref{aa} and \eqref{bb}.

\begin{theorem}\label{prop1}
For all $\mu$
\begin{equation}
(1-s_i)(F_{\mu} +F_{s_i\mu} )=0.
\label{eq1}
\end{equation}
If $\mu_i>\mu_{i+1}$
\begin{equation}
(1-s_i)(tx_{i+1}F_{\mu} +x_iF_{s_i\mu} )=0.
\label{eq2}
\end{equation}
\end{theorem}

\begin{lemma}\label{lem:base}
	\cref{prop1} is true when each $\mu_j \leq 1$.
\end{lemma}
\begin{proof}
	When each $\mu_j \leq 1$, $F_{\mu} = \prod x_j$ where the product is over all $j$ where 
	$\mu_j = 1$.  The proof is now immediate.
\end{proof}

\begin{lemma}
\cref{prop1} implies \eqref{aa} and \eqref{bb}.
\end{lemma}

\begin{proof}
If $\mu_i=\mu_{i+1}$, then $F_{s_i \mu}= F_{\mu}$, so \eqref{eq1} implies that 
$(1-s_i)F_{\mu}=0.$ This implies \eqref{bb}.

If $\mu_i>\mu_{i+1}$, by \eqref{eq2} we have that 
\[
tx_{i+1}F_{\mu}+x_iF_{s_i\mu}-tx_is_iF_{\mu}-x_{i+1}s_i F_{s_i\mu}=0.
\]
Using \eqref{eq1} to replace the quantity $s_i F_{s_i \mu}$ above, we get 
\[
tx_{i+1}F_{\mu}+x_iF_{s_i\mu}-tx_is_iF_{\mu}-x_{i+1}(F_{\mu}+F_{s_i \mu} - s_i F_{\mu})=0.
\]
This is easily seen to be equivalent to \eqref{aa}.

\end{proof}

Our next goal is to compare the quantities $F_{\mu}^{\lambda} $, $F_{s_i\mu}^{\lambda} $, $F_{\mu}^{s_i\lambda} $, $F_{s_i\mu}^{s_i\lambda} .$  Without loss of generality, we can assume that $\mu_i \geq \mu_{i+1}$ and $\lambda_i \geq \lambda_{i+1}$.
In \cref{lem:equal} we will treat the case that $\mu_i = \mu_{i+1}$, or 
$\lambda_i = \lambda_{i+1}$, and in 
\cref{lem:symm1} we will treat the case that $\mu_i > \mu_{i+1}>0$.

\begin{defn}
Let $\lambda$ and $\mu$ be weak compositions with $n$ parts.
Recall the definition of $\mathcal{Q}_{\mu}^{\lambda}$
from \cref{def:two-row}.  
Given two permutations $\pi, \sigma \in S_n$, we 
define $\phi_{\pi}^{\sigma}: 
\mathcal{Q}_{\mu}^{\lambda} \rightarrow \mathcal{Q}_{\pi\mu}^{\sigma \lambda}$ to be the map from 
$\mathcal{Q}_{\mu}^{\lambda}$ to $\mathcal{Q}_{\pi\mu}^{\sigma \lambda}$ 
which permutes the contents of the bottom and top row of the multiline queue according to $\pi$ and $\sigma$ as $\pi\mu=(\mu_{\pi(1)},\ldots,\mu_{\pi(n)})$ and $\sigma\lambda=(\lambda_{\sigma(1)},\ldots,\lambda_{\sigma(n)})$, 
while preserving the pairings between the balls. (Set $\phi_{\pi}^{\sigma}Q=\emptyset$ if the result
is not a valid  multiline queue.) 
Usually we will choose $\pi, \sigma \in \{s_i, \id, \omega\}$,
where
        $\omega \mu$ is as in \eqref{action3}.
Note that  $\phi_{s_i}^{s_i}$ is a bijection.  We also use the notation $\phi^{s_i}=\phi_{\id}^{s_i}$ and $\phi_{s_i}=\phi_{s_i}^{\id}$. See \cref{fig:Lem3.7} for an example of $\phi^{s_1}$.
\end{defn}

For ease of notation, we will identify the balls of the multiline queue $Q\in\mathcal{Q}_{\mu}^{\lambda}$ with their labels $\lambda_i\in\lambda$ and $\mu_i\in\mu$. For $Q'=\phi_{\pi\mu}^{\sigma\lambda} Q$, when we refer to the balls $\mu_i$ and $\lambda_i$, we are referring to the balls in row 1, column $\pi(i)$ and row 2, column $\sigma(i)$ of $Q'$, respectively. For instance in the example of \cref{fig:Lem3.7}, the ball $\lambda_1$ corresponds to the ball labeled 2 in the top row of column 1 of $Q_0$ and column 2 of $\phi^{s_1}Q_0$, respectively.

\begin{lemma}\label{lem:equal}
If $\mu_i=\mu_{i+1}\ge 0$, then
\[
F_{\mu}^{\lambda} =F_{s_i\mu}^{\lambda} =F_{\mu}^{s_i\lambda} =F_{s_i\mu}^{s_i\lambda} .
\]
If $\lambda_i=\lambda_{i+1}$, then
\[
F_{\mu}^{\lambda} =F_{\mu}^{s_i\lambda} \quad \mbox{and}\quad F_{s_i\mu}^{\lambda} =F_{s_i\mu}^{s_i\lambda} .
\]
\end{lemma}

\begin{proof}
The equalities when $\lambda_i=\lambda_{i+1}$ are immediate since $s_i\lambda=\lambda$. If $\mu_i=\mu_{i+1}=0$, swapping $\lambda_{i}$ and $\lambda_{i+1}$ does not change the weights of any pairings, since no balls are being skipped in columns $i,i+1$. If $\mu_i=\mu_{i+1}>0$, we use the fact that $\lambda_i \leq \mu_i$ if $\mu_i>0$, and $\lambda_{i+1} \leq \mu_{i+1}$, which means that the only possible pairings between elements in columns $i,i+1$ are trivial ones, and thus no pairings from columns $i,i+1$ are skipping over the balls $\mu_i,\mu_{i+1}$. For an example of the $\mu_i=\mu_{i+1}$ case, see \cref{fig:Lem3.7}.
\end{proof}

\begin{figure}[!ht]
  \centerline{\includegraphics[width=0.8\linewidth]{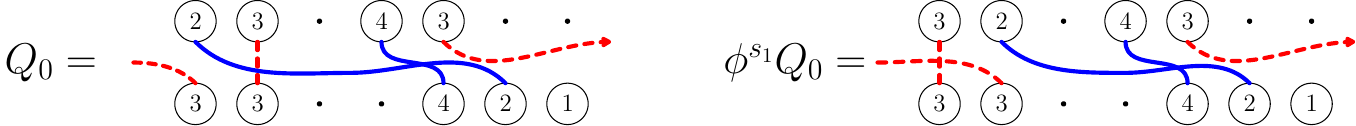}}
\centering
\caption{For $\mu=(3,3,0,0,4,2,1)$ and $\lambda=(2,3,0,4,3,0,0)$, $Q_0\in\mathcal{Q}_{\mu}^{\lambda}$ and $\phi^{s_1} Q_0$ are shown, with equal weights of their respective pairings.}\label{fig:Lem3.7}
\end{figure}

Having taken care of the cases in \cref{lem:equal}, we will now assume without loss of generality
that $\mu_i>\mu_{i+1}$ and $\lambda_i > \lambda_{i+1}$.

\begin{lemma}\label{lem:circ}
	Recall the action of $\omega$ on compositions from \eqref{action3}.
We have 
\[F_{\mu}^{\lambda}(x_1,\ldots,x_n)=q^{\max(\mu_n-1,0)-\max(\lambda_n-1,0)} F_{\omega\mu}^{\omega\lambda}(x_n,x_1,\ldots,x_{n-1}).\]
\end{lemma}

\begin{proof}
There are five cases for the last column of $Q\in\mathcal{Q}_{\mu}^{\lambda}$, which we show in \cref{fig:circ} along with the corresponding multiline queues $\phi_{\omega}^{\omega}Q$. When $\lambda_n=\mu_n$, the weights of all pairings in $Q$ vs. $\phi_{\omega}^{\omega}Q$ are identical. When $\lambda_n\neq \mu_n$, the weights of all pairings are identical except for the pairings from $\lambda_n$ and the pairings to $\mu_n$:
\begin{itemize}
\item if $0<\lambda_n<\mu_n$ we have $\wt(\phi_{\omega}^{\omega}Q)=q^{\mu_n-\lambda_n}\wt(Q)$, since the pairing to $\mu_n$ is now cycling, but the pairing from $\lambda_n$ is no longer cycling. 
\item if $\lambda_n=0$, we have $\wt(\phi_{\omega}^{\omega}Q)=q^{\mu_n-1}\wt(Q)$, since the pairing to $\mu_n$ is now cycling.
\item if $\mu_n=0$, we have $\wt(\phi_{\omega}^{\omega}Q)=q^{-(\lambda_n-1)}\wt(Q)$, since the pairing from $\lambda_n$ is no longer cycling.
\end{itemize}
Thus we get the desired equality.

\begin{figure}[!ht]
  \centerline{\includegraphics[width=\textwidth]{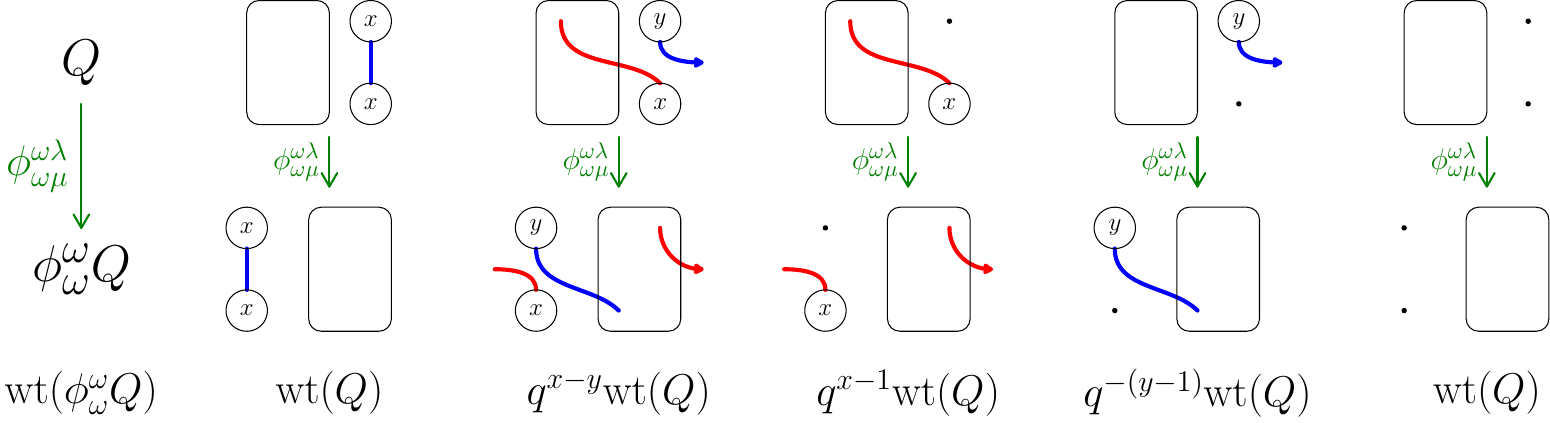}}
\centering
\caption{The five cases of the last column of $Q\in\mathcal{Q}_{\mu}^{\lambda}$: when $\mu_n=\lambda_n=x>0$, when $x=\mu_n>\lambda_n=y>0$, when $\mu_n=x$ and $\lambda_n=0$, when $\mu_n=0$ and $\lambda_n=y$, and when $\lambda_n=\mu_n=0$.}
\label{fig:circ}
 \end{figure}
 \end{proof}

\begin{lemma}\label{lem:symm1}
Suppose $\mu_i > \mu_{i+1}>0$, and $\lambda_i > \lambda_{i+1}\geq 0$.
\begin{enumerate}
\item If  $\mu_{i+1} > \lambda_i$,
		\label{symm1a}
\[
tF_{\mu}^{\lambda} =F_{s_i\mu}^{\lambda} =tF_{\mu}^{s_i\lambda} =F_{s_i\mu}^{s_i\lambda} .
\]
\item If  $\mu_{i+1}=\lambda_i$,
\[
F_{s_i\mu}^{\lambda} = tF_{\mu}^{s_i\lambda} \quad \mbox{and} 
\quad F_{\mu}^{\lambda} +F_{s_i\mu}^{\lambda} =F_{\mu}^{s_i\lambda} +F_{s_i\mu}^{s_i\lambda}. 
\]
\item If $\mu_{i+1}<\lambda_i$,
\[
F_{\mu}^{\lambda} =F_{s_i\mu}^{s_i\lambda} \quad \mbox{and} \quad F_{s_i\mu}^{\lambda} =F_{\mu}^{s_i\lambda} =0.
\]
\item If $\mu_{i+1}<\lambda_{i+1}$, 
\[
F_{\mu}^{\lambda}=F_{s_i\mu}^{s_i\lambda}=F_{\mu}^{s_i\lambda}=F_{s_i\mu}^{\lambda}=0.
\]
\end{enumerate}
\end{lemma}
\begin{proof}
\begin{figure}[!ht]
  \centerline{\includegraphics[width=\textwidth]{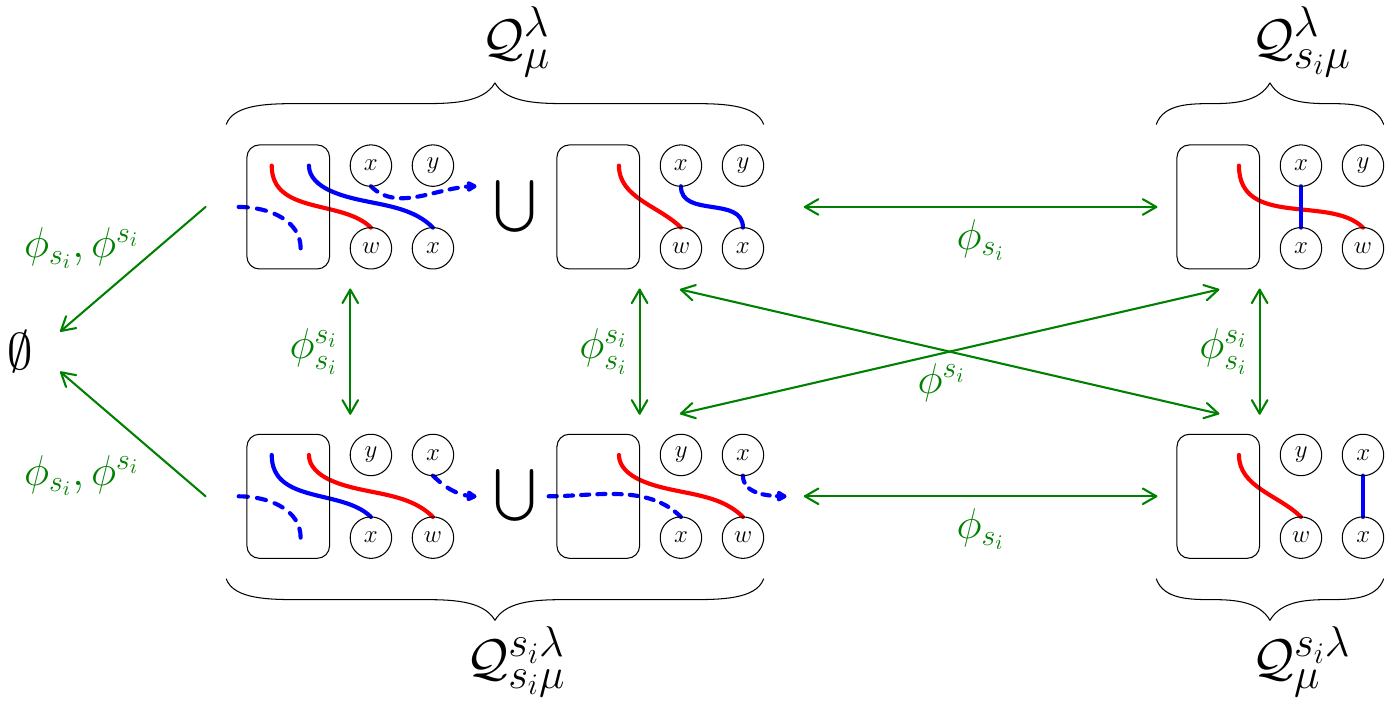}}
\centering
\caption{This diagram illustrates case (2) of \cref{lem:symm1}. Let $\mu_i=w$, $\lambda_i=\mu_{i+1}=x$, and $\lambda_{i+1}=y$, with $w > x > y \geq 0$.} 
\label{fig:phi}
 \end{figure}

Cases (1), (3), and (4) are straightforward, so we begin by taking care of these cases.
In Case (1), the maps $\phi_{s_i}$, $\phi^{s_i}$, and $\phi_{s_i}^{s_i}$ define bijections between $\mathcal{Q}_{\mu}^{\lambda}$ and the sets $\mathcal{Q}_{s_i\mu}^{\lambda}$, $\mathcal{Q}_{\mu}^{s_i\lambda}$, and $\mathcal{Q}_{s_i\mu}^{s_i\lambda}$ respectively. The only difference between the weights of the multiline queues in these four sets comes from whether or not the pairing to ball $\mu_i$ skips over the ball 
$\mu_{i+1}$.  When this pairing does skip over ball $\mu_{i+1}$,
we get an extra contribution of $t$ to the weight.
Therefore we have $tF_{\mu}^{\lambda}=tF_{\mu}^{s_i\lambda}=F_{s_i\mu}^{\lambda}=F_{s_i\mu}^{s_i\lambda}$.

In Case (3), $F_{s_i\mu}^{\lambda}=F_{\mu}^{s_i\lambda}=0$ since a larger label cannot be above a smaller one in a valid multiline queue. Thus we must show $F_{\mu}^{\lambda}=F_{s_i\mu}^{s_i\lambda}$.

If $\mu_i=\lambda_i$, the equality is immediate. Otherwise, let $Q\in \mathcal{Q}_{\mu}^{\lambda}$ be a generalized multiline queue, and let $\phi_{s_i}^{s_i}Q \in \mathcal{Q}_{s_i\mu}^{s_i\lambda}$ be the corresponding queue with the same ball pairings. In $Q$, the pairing from $\lambda_i$ skips over ball $\mu_{i+1}$, 
contributing a $t$ to $\wt(Q)$, whereas in $\phi_{s_i}^{s_i}Q$ the 
pairing to $\mu_i$ skips over $\mu_{i+1}$, contributing a $t$ to $\wt(\phi_{s_i}^{s_i}Q)$. The rest of the pairings contribute identical weights, and thus $\wt(Q)=\wt(\phi_{s_i}^{s_i}Q)$, so the equality follows.

For case (4), since we have assumed $\lambda_{i+1}<\lambda_i$, when $\mu_{i+1}<\lambda_{i+1}$, $\mathcal{Q}_{\mu}^{\lambda}=\emptyset$ by definition. Combined with the assumption $\mu_{i+1}<\mu_i$, we get the rest from Case (3).

In what follows, we will write $\lambda_i \sim \mu_{i+1}$ or $\lambda_i \not\sim \mu_{i+1}$ based on whether ball $\lambda_i$ is paired with ball $\mu_{i+1}$.

Finally consider Case (2), illustrated in \cref{fig:phi}. In this diagram, $\mathcal{Q}_{\mu}^{\lambda}$ consists of the set of multiline queues where $\lambda_i \not\sim \mu_{i+1}$ (the left multiline queue under the curly brace) and the set where $\lambda_i\sim\mu_{i+1}$ (the right multiline queue under the curly brace). The images of the maps $\phi_{s_i}, \phi^{s_i},\phi_{s_i}^{s_i}$ applied to $\mathcal{Q}_{\mu}^{\lambda}$ are then the sets $\mathcal{Q}_{s_i\mu}^{\lambda}$, $\mathcal{Q}_{\mu}^{s_i\lambda}$, and $\mathcal{Q}_{s_i\mu}^{s_i\lambda}$, respectively, with the arrows illustrating the various bijections. Thus the top row of the diagram consists of the sets of multiline queues, the sum of whose weights is $F_{\mu}^{\lambda} +F_{s_i\mu}^{\lambda}$, and the bottom row consists of the sets of multiline queues, the sum of whose weights is $F_{\mu}^{s_i\lambda} +F_{s_i\mu}^{s_i\lambda}$, and the map $\phi_{s_i}^{s_i}$ is the bijection between those sets.

We claim that to prove the lemma, it suffices to prove it for $i=n-1$.
	To see this, note that for $i<n-1$, \cref{lem:circ} implies that 
	\begin{equation*}
		\frac{F^{\lambda}_{\mu}(x_1,\dots,x_n)}{F^{\omega \lambda}_{\omega \mu}(x_n,x_1,\dots,x_{n-1})} = 
	q^{\max(\mu_n-1,0)-\max(\lambda_n-1,0)} = 
		\frac{F^{\lambda}_{s_i \mu}(x_1,\dots,x_n)}{F^{\omega \lambda}_{\omega (s_i \mu)}(x_n,x_1,\dots,x_{n-1})}.
	\end{equation*}
Therefore for $i<n-1$ we have 
	\begin{equation} \label{iter1}
		\frac{F^{\lambda}_{s_i \mu}(x_1,\dots,x_n)}{F^{\lambda}_{\mu}(x_1,\dots,x_n)} = 
		\frac{F^{\omega \lambda}_{\omega(s_i \mu)}(x_n,x_1,\dots,x_{n-1})}{F^{\omega \lambda}_{\omega \mu}(x_n,x_1,\dots,x_{n-1})} = 
		\frac{F^{\omega \lambda}_{s_{i+1} \omega(\mu)}(x_n,x_1,\dots,x_{n-1})}{F^{\omega \lambda}_{\omega \mu}(x_n,x_1,\dots,x_{n-1})}.
	\end{equation}
Similarly for $i<n-1$ we have 
	\begin{equation}\label{iter2}
		\frac{F^{s_i \lambda}_{\mu}(x_1,\dots,x_n)}{F^{\lambda}_{\mu}(x_1,\dots,x_n)} = 
		\frac{F^{s_{i+1} \omega (\lambda)}_{\omega(\mu)}(x_n,x_1,\dots,x_{n-1})}{F^{\omega \lambda}_{\omega \mu}(x_n,x_1,\dots,x_{n-1})} \text{ and }
	\end{equation}
	\begin{equation}\label{iter3}
		\frac{F^{s_i \lambda}_{s_i \mu}(x_1,\dots,x_n)}{F^{\lambda}_{\mu}(x_1,\dots,x_n)} = 
		\frac{F^{s_{i+1} \omega (\lambda)}_{s_{i+1} \omega(\mu)}(x_n,x_1,\dots,x_{n-1})}{F^{\omega \lambda}_{\omega \mu}(x_n,x_1,\dots,x_{n-1})}, 
	\end{equation}
	so by iterating \eqref{iter1}, \eqref{iter2}, and \eqref{iter3},
	we can reduce the proof of the lemma to 
the case $i=n-1$.

When $i=n-1$ and $i+1=n$, the transposition affects only the rightmost two columns.  Write
$\mu_{i+1}=\lambda_i=x$ and consider $Q\in \mathcal{Q}_{\mu}^{\lambda}$.

\begin{enumerate}
\item 
\label{c0} 
Observe that $\wt(\phi_{s_i}Q)=t\wt(\phi^{s_i}Q)$ when $\lambda_i\sim \mu_{i+1}$ because the pairing to ball $\mu_{i}$ in $\phi_{s_i}Q$ skips over ball $\mu_{i+1}$, contributing an extra $t$. This proves the equality
$F_{s_i\mu}^{\lambda} = tF_{\mu}^{s_i\lambda}$.
\item 
\label{c1} When $\lambda_i\not\sim \mu_{i+1}$ in $Q$, we have $\wt(Q)=\wt(\phi_{s_i}^{s_i}Q)$. This is because in $Q$, the pairing from ball $\lambda_i$ obtains an extra $t$ by skipping over ball $\mu_{i+1}$, whereas in $\phi_{s_i}^{s_i}Q$ 
the pairing to ball $\mu_i$ skips over ball $\mu_{i+1}=x$.

\item 
\label{c2} Now consider $\phi^{s_i}Q$. 
This is only nonempty if in $Q$, $\lambda_i\sim \mu_{i+1}$.  Moreover 
$\phi^{s_i}$ defines a bijection from $\{Q \ \vert \ Q\in \mathcal{Q}_{\mu}^{\lambda}, \lambda_i \sim \mu_{i+1}\}$
to $\mathcal{Q}_{\mu}^{s_i \lambda}$.  So consider $Q$ where $\lambda_i \sim \mu_{i+1}$.

Let $f$ be the number of free balls remaining in $Q$ right before we pair the ball $\lambda_i$.  The weight of the pairing
$\lambda_i \sim \mu_{i+1}$ in $Q$ is 
$\frac{(1-t)}{1-q^{x-1}t^{f}}$. 
Since $i$ and $i+1$ are rightmost, ball $\lambda_i$ is the first instance of label $x$ to be paired. Thus every other pairing in $Q$
gets the same weight as the corresponding pairing in $\phi^{s_i} Q$, and so
 $\wt(Q)=\wt(\phi^{s_i}Q)
\frac{(1-t)}{1-q^{x-1}t^{f}}$. 

\item 
\label{c3} Similarly, when $\lambda_i \sim \mu_{i+1}$, we have
 $\wt(\phi_{s_i}^{s_i}Q)=\wt(\phi_{s_i}Q)\frac{q^{x-1}t^{f-1}(1-t)}{1-q^{x-1}t^{f}}$, since the pairing 
in $\phi_{s_i}^{s_i} Q$ from ball $\lambda_i$ to ball $\mu_{i+1}$ cycles and skips all the free balls except for ball $\mu_{i+1}$, hence contributing $t^{f-1}$. 
By \cref{c0},  we have $\wt(\phi_{s_i}^{s_i}Q)=\wt(\phi^{s_i}Q)\frac{q^{x-1}t^{f}(1-t)}{1-q^{x-1}t^{f}}$.

\item 
\label{c4} By \cref{c2} and \cref{c3}, for $Q\in \mathcal{Q}_{\mu}^{\lambda}$ with 
$\lambda_i \sim \mu_{i+1}$, we have
$\wt(\phi_{s_i}^{s_i} Q) = q^{x-1} t^f \wt(Q)$.

\end{enumerate}

Let us now write down the proof:
\begin{align*}
F_{\mu}^{\lambda}-F_{s_i\mu}^{s_i\lambda}&=\sum_{Q\in\mathcal{Q}_{\mu}^{\lambda}} \Big(\wt(Q)-\wt(\phi_{s_i}^{s_i}Q)\Big)\\
&=\sum_{\substack{Q\in\mathcal{Q}_{\mu}^{\lambda},\\\ \lambda_i \sim \mu_{i+1}}} \Big(\wt(Q)-\wt(\phi_{s_i}^{s_i}Q)\Big)+\sum_{\substack{Q\in\mathcal{Q}_{\mu}^{\lambda},\\\lambda_i \not\sim \mu_{i+1}}} \Big(\wt(Q)-\wt(\phi_{s_i}^{s_i}Q)\Big)\\
&=\sum_{\substack{Q\in\mathcal{Q}_{\mu}^{\lambda},\\\lambda_i \sim \mu_{i+1}}} \wt(Q)(1-q^{x-1}t^{f})\\
&=\sum_{Q\in\mathcal{Q}_{\mu}^{s_i\lambda}} \wt(Q) (1-t) \\ 
&=F_{\mu}^{s_i\lambda}-F_{s_i\mu}^{\lambda}.
\end{align*}
Here the equality between the second and third line follows from Items 5 and 2,
and the equality between the third and fourth line follows from Item 3.
The last one is a consequence of Item 1.
\end{proof}

A direct consequence of \cref{lem:equal} and \cref{lem:symm1} is:

\begin{lemma}\label{lem:symmetry}
If $\mu_i,\mu_{i+1}>0$ or $\mu_i=\mu_{i+1}$  then
\[
F_{\mu}^{\lambda} +F_{s_i\mu}^{\lambda} =F_{\mu}^{s_i\lambda} +F_{s_i\mu}^{s_i\lambda}. 
\]
\end{lemma}

Now we consider the case that $\mu_i > \mu_{i+1} = 0$.  Without loss of generality 
we assume $\lambda_i \geq \lambda_{i+1}$.

\begin{lemma}\label{morecases}
Suppose that 
 $\mu_i>\mu_{i+1}=0$ and $\lambda_i \geq \lambda_{i+1}$.
Then we have the following:
\begin{enumerate}
\item \label{item1} If $\lambda_i=\lambda_{i+1}$ or $\mu_i> \lambda_i,\lambda_{i+1}$
then
\[
x_{i+1}F_{\mu}^{\lambda} =x_iF_{s_i\mu}^{\lambda} =x_{i+1}F_{\mu}^{s_i\lambda} =x_iF_{s_i\mu}^{s_i\lambda} .
\]
In particular, 
both $F_{\mu}^{\lambda} +F_{s_i \mu}^{\lambda} $
and  
$F_{\mu}^{s_i \lambda} +F_{s_i \mu}^{s_i \lambda} $
are symmetric in $x_i$ and $x_{i+1}$.
\item  \label{item2}
If $\mu_{i}=\lambda_i>\lambda_{i+1}$ then
\reqnomode
		\begin{equation}\label{item2-1}
			tx_{i+1}F_{\mu}^{\lambda} +x_iF_{s_i\mu}^{\lambda} =
tx_{i+1}F_{\mu}^{s_i\lambda} +x_iF_{s_i\mu}^{s_i\lambda} .
		\end{equation}
		We also have that 
		$x_{i+1} 
		F_{\mu}^{\lambda} =x_iF_{s_i\mu}^{s_i\lambda} ,$ and 
		\begin{equation} \label{item2-2}
		tx_{i+1}F_{\mu}^{s_i\lambda} +(1-t)x_{i+1}F_{\mu}^{\lambda}=x_iF_{s_i\mu}^{\lambda} . \ \ 
		\end{equation}
\item \label{item3} If $\lambda_i>\mu_i \geq \lambda_{i+1}$ then 
$$
x_i F_{s_i\mu}^{\lambda} = t x_{i+1} 
F_{\mu}^{s_i\lambda} ; \ \ F_{\mu}^{\lambda} =F_{s_i\mu}^{s_i\lambda} =0.
$$
\item \label{item4} If $\lambda_i > \lambda_{i+1} > \mu_i$ then 
$$F_{\mu}^{\lambda} =F_{s_i\mu}^{s_i\lambda} = F_{\mu}^{s_i\lambda} =F_{s_i\mu}^{\lambda} =0.$$
\end{enumerate}
\end{lemma}

\begin{proof}
\cref{item1}, \cref{item3}, and \cref{item4} follow easily from the definitions, as does the statement
$	x_{i+1} 
		F_{\mu}^{\lambda} =x_iF_{s_i\mu}^{s_i\lambda} $ from \cref{item2}.
		The proof of \eqref{item2-1} is completely analogous to the proof of 
	Case (2) of \cref{lem:symm1}.  Meanwhile
		 \eqref{item2-2} follows from \eqref{item2-1} together with the fact that 
		$x_{i+1} 
		F_{\mu}^{\lambda} =x_iF_{s_i\mu}^{s_i\lambda} .$ 

\end{proof}

The following lemma is a direct consequence of \cref{morecases}.
\begin{lemma}
Suppose that 
 $\mu_i>\mu_{i+1}=0$.
	Then we have \begin{equation}\label{tidentity}
	tx_{i+1}F_{\mu}^{\lambda} +x_iF_{s_i\mu}^{\lambda} =
tx_{i+1}F_{\mu}^{s_i\lambda} +x_iF_{s_i\mu}^{s_i\lambda} .
	\end{equation}
\end{lemma}

In \cref{prop:3.14} through \cref{prop:3.18} below, we will prove \eqref{eq1} and \eqref{eq2} together, by induction
on the number of rows $L$ in the diagrams (equivalently, on the value $L$ of the largest part 
in the composition $\mu$).  The base case $L=1$ is covered by 
\cref{lem:base}.  For fixed $L \geq 2$, we will be assuming that all cases of 
\eqref{eq1} and \eqref{eq2} are true for diagrams with at most $L-1$ rows.

\begin{proposition}\label{prop:3.14} Let $L>1$.
Suppose that \eqref{eq1} holds for compositions with maximal part at most $L-1$. Then
	\eqref{eq1} holds for compositions $\mu$ with maximal part $L$ and such that $\mu_i=\mu_{i+1}$;
	in other words,
	$F_{\mu} $ is symmetric in $x_i$ and $x_{i+1}$.
\end{proposition}
\begin{proof}
We compute
\begin{eqnarray*}
2F_{\mu} &=&\sum_{\lambda} \big( F_{\mu}^{\lambda} F_{\lambda^-} + F_{\mu}^{s_i\lambda} F_{s_i\lambda^-}\big) \\
&=&\sum_{\lambda} F_{\mu}^{\lambda}  (F_{\lambda^-} +F_{s_i\lambda^-} ).
\end{eqnarray*}
The first equality comes from \cref{lem:recursive}, and the second comes from \cref{lem:equal}, which says that 
$F_{\mu}^{\lambda} =F_{\mu}^{s_i\lambda} $ when $\mu_i=\mu_{i+1}$. 

But now we have that 
$(F_{\lambda^-} +F_{s_i\lambda^-} )$ is symmetric in $x_i$ and $x_{i+1}$ by induction, and 
$F_{\mu}^{\lambda} $ is symmetric in $x_i$ and $x_{i+1}$ by definition (since $\mu_i=\mu_{i+1}$, the variables $x_i$ and $x_{i+1}$ appear in the
$\mathbf x$-weight 
as either $1$ or $x_{i}x_{i+1}$, depending on whether $\mu_i=0$ or not, and only $\mu$ contributes to the $\mathbf x$-weight of $F_{\mu}^{\lambda}$).
This implies that $F_{\mu} $ is symmetric in $x_i$ and $x_{i+1}$.
\end{proof}

\begin{proposition}\label{prop:3.15}
Suppose that \eqref{eq1} holds for compositions with maximal part at most $L-1$.
	Then
	\eqref{eq1} holds for compositions $\mu$ with maximal part $L$ and such that $\mu_i>\mu_{i+1}>0$.
\end{proposition}
\begin{proof} We have that 
\begin{eqnarray*}
2(F_{\mu} +F_{s_i\mu} )
	&=&\sum_\lambda \Big( (F_{\mu}^{\lambda} +F_{s_i\mu}^{\lambda} )F_{\lambda^-} + (F_{\mu}^{s_i\lambda} +F_{s_i\mu}^{s_i\lambda} )F_{s_i\lambda^-} \Big)\\
	&=&\sum_\lambda (F_{\mu}^{\lambda} +F_{s_i\mu}^{\lambda} )(F_{\lambda^-} + F_{s_i\lambda^-} )\\
	&=&\sum_\lambda (F_{\mu}^{\lambda} +F_{s_i\mu}^{\lambda} )s_i(F_{\lambda^-} + F_{s_i\lambda^-} )\\
	&=&\sum_\lambda s_i(F_{\mu}^{\lambda} +F_{s_i\mu}^{\lambda} )s_i(F_{\lambda^-} + F_{s_i\lambda^-} )\\
	&=&s_i\sum_\lambda (F_{\mu}^{\lambda} +F_{s_i\mu}^{\lambda} )(F_{\lambda^-} + F_{s_i\lambda^-} )\\
	&=&s_i\sum_\lambda \Big((F_{\mu}^{\lambda}  + F_{s_i\mu}^{\lambda} )F_{\lambda^-} + (F_{\mu}^{s_i\lambda}  +
	F_{s_i\mu}^{s_i\lambda} )F_{s_i\lambda^-} \Big)\\
&=& 2 s_i (F_{\mu} +F_{s_i\mu} ).
\end{eqnarray*}
	The first equality comes from \cref{lem:recursive}. The second is due to \cref{lem:symmetry}. The third uses the induction step. The fourth one uses the (trivial) fact that $s_i (F_{\mu}^{\lambda} ) = F_{\mu}^{\lambda} $
	whenever $\mu_i$ and $\mu_{i+1}$ are both nonzero.
\end{proof}

\begin{proposition}
\label{prop:3.16}
Suppose that \eqref{eq1} and \eqref{eq2} hold 
for compositions with maximal part at most $L-1$.
Then
\eqref{eq1} holds for compositions $\mu$ with maximal part $L$ and such that $\mu_i>\mu_{i+1}=0$.
\end{proposition}
\begin{proof}
	We have that 
\begin{eqnarray*}
F_{\mu} +F_{s_i\mu} 
	&=&\sum_{\lambda_i>\lambda_{i+1}} \left( (F_{\mu}^{\lambda} +F_{s_i\mu}^{\lambda} )F_{\lambda^-} + 
(F_{\mu}^{s_i\lambda} +
	F_{s_i\mu}^{s_i\lambda} )F_{s_i\lambda^-} \right) \\
 &&+\sum_{\lambda_i=\lambda_{i+1}} (F_{\mu}^{\lambda} +F_{s_i\mu}^{\lambda} )F_{\lambda^-}.\\
\end{eqnarray*}
By \cref{item1} of \cref{morecases} and the induction hypothesis, the term on the right-hand side 
where $\lambda_i = \lambda_{i+1}$ is symmetric in $x_i$ and $x_{i+1}$.
We need to show that the same is true for the rest of the right-hand side.

Using \cref{morecases}, we have 
that 
\[
	\sum_{\lambda_i>\lambda_{i+1}}\left( (F_{\mu}^{\lambda} +F_{s_i\mu}^{\lambda} )F_{\lambda^-} + (F_{\mu}^{s_i\lambda} +
	F_{s_i\mu}^{s_i\lambda} )F_{s_i\lambda^-} \right)
\] 
is equal to 

\begin{align}
	& \sum_{\mu_i > \lambda_i>\lambda_{i+1}} \label{first}
(F_{\mu}^{\lambda} + F_{s_i \mu}^{\lambda})(F_{\lambda^-} +F_{s_i\lambda^-} )\\
	&+ \sum_{\mu_i=\lambda_{i}>\lambda_{i+1}} \label{second}
\Big(
\big( F_\mu^\lambda  F_{\lambda^-} +F_{s_i\mu}^{s_i\lambda} F_{s_i\lambda^-} \big)+
\big( 
	F_{\mu}^{s_i\lambda} F_{s_i\lambda^-} 
	+
	F_{s_i\mu}^{\lambda} F_{\lambda^-} 
	\big)
	\Big)\\
&+ \sum_{\lambda_{i}>\mu_i\ge \lambda_{i+1}}
	\big( F_{s_i \mu}^{\lambda} F_{\lambda^-} +F_{\mu}^{s_i \lambda}F_{s_i\lambda^-} \big) . \label{third}
\end{align}
	By induction and \cref{item1} of \cref{morecases}, \eqref{first}
	 is symmetric in $x_i$ and $x_{i+1}$. 
	Meanwhile \eqref{third} is equal to 
$$ \sum_{\lambda_{i}>\mu_i\ge \lambda_{i+1}}
\frac{F_{s_i \mu}^{\lambda}}{t x_{i+1}} (tx_{i+1} F_{\lambda^-} +x_i F_{s_i\lambda^-} ),$$
	which by induction is also symmetric in $x_i$ and $x_{i+1}$.

	Finally we use \cref{item2} of \cref{morecases} to rewrite \eqref{second}	
	as
	\begin{align*}
&\hspace{-0.2in}
		\sum_{\mu_i = \lambda_{i}>\lambda_{i+1}} \Big(F_{\mu}^{\lambda}  F_{\lambda^-} + 
		\frac{x_{i+1}}{x_i} F_{\mu}^{\lambda} F_{{s_i \lambda}^-} + 
		F_{\mu}^{s_i \lambda} F_{(s_i \lambda)^-} + \frac{t x_{i+1}}{x_i} F_{\mu}^{s_i \lambda} F_{\lambda^-} + 
		\frac{(1-t) x_{i+1}}{x_i} F_{\mu}^{\lambda} F_{\lambda^-}\Big) \\
		& =   \sum_{\mu_i=\lambda_{i}>\lambda_{i+1}}
		\frac{F_{\mu}^{s_i \lambda}}{x_i} \big(t x_{i+1} F_{\lambda^-} + x_i F_{(s_i \lambda)^-}\big)  
		 +   \sum_{\mu_i=\lambda_{i}>\lambda_{i+1}}
		\frac{F_{\mu}^{\lambda}}{x_i} (x_i + x_{i+1}) (F_{\lambda^-} + F_{(s_i \lambda)^-}) \\
		& -   \sum_{\mu_i=\lambda_{i}>\lambda_{i+1}}
		\frac{F_{\mu}^{\lambda}}{x_i} \big(t x_{i+1} F_{\lambda^-} + x_i F_{(s_i \lambda)^-}\big).
\end{align*}
By induction all parts are symmetric in $x_i$ and $x_{i+1}$.
\end{proof}

\begin{proposition}
\label{prop:3.17}
Suppose that \eqref{eq1} and \eqref{eq2} hold 
for compositions with maximal part at most $L-1$.
Then
\eqref{eq2} holds for compositions $\mu$ with maximal part $L$ and such that $\mu_i>\mu_{i+1}>0$.
\end{proposition}
\begin{proof}
	We need to show that $tx_{i+1}F_{\mu}+x_iF_{s_i\mu}$ is symmetric in $x_i$ and $x_{i+1}$.
Towards this end, we write 
	\begin{equation}\label{eq:bigsum}
tx_{i+1}F_{\mu}+x_iF_{s_i\mu} 
		= \sum_{\lambda_i=\lambda_{i+1}} (tx_{i+1}F_{\mu}^{\lambda}+x_iF_{s_i\mu}^{\lambda})F_{\lambda^-} + \sum_{\lambda_i\neq \lambda_{i+1}} (tx_{i+1}F_{\mu}^{\lambda}+x_iF_{s_i\mu}^{\lambda})F_{\lambda^-}. 
	\end{equation}
	In the first sum on the right-hand side of \eqref{eq:bigsum}, where $\lambda_i=\lambda_{i+1}$, we have
\[
(tx_{i+1}F_{\mu}^{\lambda}+x_iF_{s_i\mu}^{\lambda})F_{\lambda^-}= (tx_{i+1}F_{\mu}^{\lambda}+x_i(tF_{\mu}^{\lambda}))F_{\lambda^-} = t(x_i+x_{i+1})F_{\mu}^{\lambda}F_{\lambda^-}.
\]
	Note that $F_{\lambda^-}$ is symmetric in $x_i$ and $x_{i+1}$ by induction (\cref{eq1}), so every such term in the first sum of \eqref{eq:bigsum} is also symmetric in $x_i$ and $x_{i+1}$.

	We write the second sum on the right-hand side of \eqref{eq:bigsum} as
\begin{align}
&\hspace{-1in}\sum_{\lambda_i\neq \lambda_{i+1}} (tx_{i+1}F_{\mu}^{\lambda}+x_iF_{s_i\mu}^{\lambda})F_{\lambda^-}\nonumber\\
\qquad=& \sum_{\lambda_i> \lambda_{i+1}} \Big((tx_{i+1}F_{\mu}^{\lambda}+x_iF_{s_i\mu}^{\lambda})F_{\lambda^-}+(tx_{i+1}F_{\mu}^{s_i\lambda}+x_iF_{s_i\mu}^{s_i\lambda})F_{s_i\lambda^-}\Big)\nonumber\\ 
 =&  \sum_{\lambda_i>\mu_{i+1}\geq \lambda_{i+1}} \Big((tx_{i+1}F_{\mu}^{\lambda}+x_iF_{s_i\mu}^{\lambda})F_{\lambda^-}+(tx_{i+1}F_{\mu}^{s_i\lambda}+x_iF_{s_i\mu}^{s_i\lambda})F_{s_i\lambda^-}\Big)\label{eq:A1}\\
 & +  \sum_{\mu_{i+1}>\lambda_i> \lambda_{i+1}} \Big((tx_{i+1}F_{\mu}^{\lambda}+x_iF_{s_i\mu}^{\lambda})F_{\lambda^-}+(tx_{i+1}F_{\mu}^{s_i\lambda}+x_iF_{s_i\mu}^{s_i\lambda})F_{s_i\lambda^-}\Big)\label{eq:A2}\\
 & +  \sum_{\lambda_i=\mu_{i+1}> \lambda_{i+1}} \Big((tx_{i+1}F_{\mu}^{\lambda}+x_iF_{s_i\mu}^{\lambda})F_{\lambda^-}+(tx_{i+1}F_{\mu}^{s_i\lambda}+x_iF_{s_i\mu}^{s_i\lambda})F_{s_i\lambda^-}\Big).\label{eq:A3}
 \end{align}
               
Note that the sums in \eqref{eq:A1}, \eqref{eq:A2}, and \eqref{eq:A3} include all terms in the original sum due to item (4) of \cref{lem:symm1}. 

For the terms in the sum of \eqref{eq:A1}, when $\lambda_i>\mu_{i+1}\geq \lambda_{i+1}$ we have
 \begin{align*}
 (tx_{i+1}F_{\mu}^{\lambda}+x_iF_{s_i\mu}^{\lambda})F_{\lambda^-}+(tx_{i+1}F_{\mu}^{s_i\lambda}+x_iF_{s_i\mu}^{s_i\lambda})F_{s_i\lambda^-} &= tx_{i+1}F_{\mu}^{\lambda}F_{\lambda^-}+x_iF_{s_i\mu}^{s_i\lambda}F_{s_i\lambda^-}\\
   &=  F_{\mu}^{\lambda}(tx_{i+1}F_{\lambda^-}+x_iF_{s_i\lambda^-}),
   \end{align*}
which is symmetric in $x_i$ and $x_{i+1}$  by induction using \eqref{eq2}.

For the terms in the sum of \eqref{eq:A2}, when $\mu_{i+1}>\lambda_i> \lambda_{i+1}$ we have
\begin{align*}
&\hspace{-0.8in}(tx_{i+1}F_{\mu}^{\lambda}+x_iF_{s_i\mu}^{\lambda})F_{\lambda^-}+(tx_{i+1}F_{\mu}^{s_i\lambda}+x_iF_{s_i\mu}^{s_i\lambda})F_{s_i\lambda^-} \\
&=  (tx_{i+1}F_{\mu}^{\lambda}+x_i(tF_{\mu}^{\lambda}))F_{\lambda^-}+(tx_{i+1}F_{\mu}^{\lambda}+x_i(tF_{\mu}^{\lambda}))F_{s_i\lambda^-}\\
 &=  tF_{\mu}^{\lambda}(x_i+x_{i+1})(F_{\lambda^-}+F_{s_i\lambda^-}),
 \end{align*}
 which is symmetric in $x_i$ and $x_{i+1}$ by induction using \eqref{eq1}.
 
 Finally, for the terms in the sum of \eqref{eq:A3}, when $\lambda_i=\mu_{i+1}> \lambda_{i+1}$ we have
\begin{align*}
&\hspace{-0.1in}  (tx_{i+1}F_{\mu}^{\lambda}+x_iF_{s_i\mu}^{\lambda})F_{\lambda^-}+(tx_{i+1}F_{\mu}^{s_i\lambda}+x_iF_{s_i\mu}^{s_i\lambda})F_{s_i\lambda^-}\\ 
 &= (tx_{i+1}F_{\mu}^{\lambda}F_{\lambda^-}+x_iF_{s_i\mu}^{s_i\lambda}F_{s_i\lambda^-})+(tx_{i+1}F_{\mu}^{s_i\lambda}F_{s_i\lambda^-}+x_iF_{s_i\mu}^{\lambda}F_{\lambda^-})\\ 
 &= \left(tx_{i+1}F_{\mu}^{\lambda}F_{\lambda^-}+x_i\left(F_{\mu}^{\lambda}+F_{s_i\mu}^{\lambda}\left(1-\frac{1}{t}\right)\right)F_{s_i\lambda^-}\right)+\left(tx_{i+1}\left(\frac{1}{t}F_{s_i\mu}^{\lambda}\right)F_{s_i\lambda^-}+x_iF_{s_i\mu}^{\lambda}F_{\lambda^-}\right)\\ 
 &= F_{\mu}^{\lambda}(tx_{i+1}F_{\lambda^-}+x_iF_{s_i\lambda^-})+F_{s_i\mu}^{\lambda}(x_i+x_{i+1})(F_{\lambda^-}+F_{s_i\lambda^-})-\frac{1}{t} F_{s_i\mu}^{\lambda}( tx_{i+1}F_{\lambda^-}+x_iF_{s_i \lambda^-}),
\end{align*}
in which all terms are symmetric in $x_i $ and $x_{i+1}$ by induction using \eqref{eq1} and \eqref{eq2}.
\end{proof}


\begin{proposition}
	\label{prop:3.18}
Suppose that \eqref{eq1} and \eqref{eq2} hold 
for compositions with maximal part at most $L-1$.
Then
\eqref{eq2} holds for compositions $\mu$ with maximal part $L$ and such that $\mu_i>\mu_{i+1}=0$.
\end{proposition}
\begin{proof}
We need to show that $tx_{i+1}F_{\mu}+x_iF_{s_i\mu}$ is symmetric in $x_i$ and $x_{i+1}$.
We have that 
\begin{align*}
&\hspace{-0.4in}2(tx_{i+1}F_{\mu} +x_iF_{s_i\mu} )\\
	&=\sum_\lambda \Big( (tx_{i+1}F_{\mu}^{\lambda} +x_iF_{s_i\mu}^{\lambda} )F_{\lambda^-} + (tx_{i+1}F_{\mu}^{s_i\lambda} 
	+x_iF_{s_i\mu}^{s_i\lambda} )F_{s_i\lambda^-} \Big) \\
&=\sum_\lambda (tx_{i+1}F_{\mu}^{\lambda} +x_iF_{s_i\mu}^{\lambda} )
	(F_{\lambda^-} +
	F_{s_i\lambda^-} ),
\end{align*}
	where we used \eqref{tidentity} in the second equality  above.
	Since $F_{\mu}^{\lambda} $ is $x_i$ times a rational function in the variables other than $x_i,x_{i+1}$,
while  $F_{s_i \mu}^{\lambda} $ is $x_{i+1}$ times a rational function in the variables other than $x_i,x_{i+1}$,
	it follows immediately that $tx_{i+1}F_{\mu}^{\lambda}+x_i F_{s_i \mu}^{\lambda}$ is symmetric in $x_i$ and $x_{i+1}$.  Using this fact and the induction hypothesis (\cref{eq1}), the right-hand side above is symmetric in $x_i$ and $x_{i+1}$.
\end{proof}

In summary, we have proved 
\eqref{eq1} and \eqref{eq2} together by induction
on the number of rows $L$ in the diagrams (equivalently, on the value $L$ of the largest part 
in the composition $\mu$).  
\cref{prop:3.14}, \cref{prop:3.15}, and \cref{prop:3.16} proved \eqref{eq1},
while \cref{prop:3.17} and \cref{prop:3.18} proved \eqref{eq2}.  
This completes our proof of 
 \eqref{aa} and \eqref{bb}.

\section{Comparing our formula to other 
formulas for Macdonald polynomials}\label{sec:comparison}

In this paper we 
 used multiline queues to give a new combinatorial formula for the Macdonald polynomial $P_{\lambda}$ 
 and the nonsymmetric Macdonald polynomial
 $E_{\lambda}$ when $\lambda$ is a partition.  
 We note that these new combinatorial formulas are quite different from the combinatorial formulas
 given by Haglund-Haiman-Loehr \cite{HHL1, HHL2, HHL3}, or 
 Ram-Yip \cite{RamYip}, or Lenart \cite{Lenart}.

While it is not obvious combinatorially, 
we show algebraically in \cref{prop:PB} that the 
polynomials $F_{\mu}$ (for $\mu$ an arbitrary composition) are equal to certain \emph{permuted basement
Macdonald polynomials}.  
Permuted-basement Macdonald polynomials $E_{\alpha}^{\sigma}({\mathbf x}; q, t)$ were introduced in \cite{thesis} and further studied in \cite{A} as a generalization of nonsymmetric Macdonald polynomials (where $\sigma \in S_n$ and $\alpha$ is a composition with $n$ parts). They have the property that the nonsymmetric Macdonald polynomial $E_{\mu}$ is equal to $E_{\rev(\mu)}^{w_0}$, where $\rev(\mu)$ denotes the reverse composition $(\mu_n, \mu_{n-1},\dots,\mu_1)$ of $\mu=(\mu_1,\dots,\mu_n)$ and $w_0$ denotes the
longest permutation $(n,\ldots,2,1)$ (written in one-line notation).  See 
\cref{rem:pb} for the definition of permuted basement Macdonald polynomials.

\begin{prop}\label{prop:PB}
For $\mu=(\mu_1,\ldots,\mu_n)$, define $\inc(\mu)$ to be the sorting of the parts of $\mu$ in increasing order. Then
\[F_{\mu}=E_{\inc(\mu)}^{\sigma}
\]
where 
$\sigma$ is the permutation of longest length such that 
	$\mu_{\sigma(1)} \leq \mu_{\sigma(2)} \leq 
	\dots \leq \mu_{\sigma(n)}$.
\end{prop}

\begin{proof}
We prove this result by reverse induction on the length of $\sigma$,
with the case that $\mu$ is a partition and $\sigma = w_0$ being the base case.  For the 
base case, when $\mu$ is a partition, 
 \cref{prop:nonsymmetric} implies that 
 $F_{\mu} = E_{\mu}= E_{\inc(\mu)}^{w_0}$.

Suppose the proposition is true for $\mu$ and $\sigma$ when $\sigma$ has length at least $r+1$.
Consider $\mu$ and $\sigma$ such that $\sigma$ has length $r$.
Find adjacent positions $i, i+1$ such that $\mu_i<\mu_{i+1}$ and let $\mu' = s_i \mu = 
(\mu_1,\dots, \mu_{i-1}, \mu_{i+1}, \mu_i, \mu_{i+2},\dots, \mu_n)$.
Let $\sigma'$ be the permutation of longest length such that 
	$\mu'_{\sigma'(1)} \leq \mu'_{\sigma'(2)} \leq 
	\dots \leq \mu'_{\sigma'(n)}$.
	 Then  $\sigma' s_i = \sigma$, where $\sigma' s_i$ is the 
	 permutation obtained from $\sigma'$ by swapping the letters
	 $i$ and $i+1$.  Moreover
	 the length of $\sigma'$ is $r+1$ and $\inc(\mu') = \inc(\mu)$ so 
	by the induction hypothesis, 
$F_{\mu'}=E_{\inc(\mu)}^{\sigma'}.$
	By \cref{thm:123}, $F_{\mu} = T_i F_{\mu'}$.  To prove the result, 
	it suffices  to show that 
	$E^{\sigma}_{\inc(\mu)} = T_i E^{\sigma'}_{\inc(\mu)}$.

To prove this claim, we use the result 
from \cite[Proposition 15]{A} that 
when  $\eta$ is an anti-partition (i.e. its parts are in increasing order)
and the length of $\sigma s_i$ is less than the length of $\sigma$,
\[
T_i E_{\eta}^{\sigma}=\begin{cases}E_{\eta}^{\sigma s_i} & \eta_{\sigma^{-1}(i)}>\eta_{\sigma^{-1}(i+1)}\\
t E_{\eta}^{\sigma s_i} & \eta_{\sigma^{-1}(i)}\le \eta_{\sigma^{-1}(i+1)}.
\end{cases}
\]
(Note that $T_i$ is denoted by $\tilde\theta_i$ in \cite{A}.)
	Applying the result to analyze $T_i E_{\eta}^{\sigma'}$ with 
	$\eta=\inc(\mu)$, we have 
	$\eta_{\sigma'^{-1}(i)} > \eta_{\sigma'^{-1}(i+1)}$, so 
	 $T_i E_{\eta}^{\sigma'}=E_{\eta}^{\sigma' s_i}$.
	 Since $\sigma' s_i = \sigma$, this proves the claim.
\end{proof}

\begin{example}
Let $\mu = (2,3,1,2,2,1)$, so that $\inc(\mu) = (1,1,2,2,2,3)$ and $\sigma = (6,3,5,4,1,2)$.
To prove that $F_{\mu} = E_{(1,1,2,2,2,3)}^{(6,3,5,4,1,2)}$ we 
start with the base case 
$$F_{(3,2,2,2,1,1)} = 
E_{(1,1,2,2,2,3)}^{(6,5,4,3,2,1)}$$ 
and then 
apply operators $T_1$, $T_4$, then $T_3$.  
We inductively obtain 
$F_{(2,3,2,2,1,1)} = 
E_{(1,1,2,2,2,3)}^{(6,5,4,3,1,2)}$, then 
$F_{(2,3,2,1,2,1)} = 
E_{(1,1,2,2,2,3)}^{(6,4,5,3,1,2)}$, then 
$F_{(2,3,1,2,2,1)} = 
E_{(1,1,2,2,2,3)}^{(6,3,5,4,1,2)}$, as desired.
\end{example}

The permuted basement Macdonald polynomials can be described 
combinatorially using \emph{nonattacking fillings} of certain diagrams 
\cite{thesis, A} 
which we call \emph{permuted basement tableaux}
(the reference \cite{thesis} cites
personal communication with Haglund for their introduction).
Note that 
these permuted basement tableaux generalize the nonattacking fillings from 
\cite{HHL3}.
In light of this, one may wonder if there is a bijection between 
multiline queues and these permuted basement tableaux.  As we explain in \cref{rem:queue}, this is 
the case when the compositions have distinct parts.  
However, for general compositions, 
the number of permuted basement tableaux is  different than the number of multiline queues. There are more permuted basement tableaux (See Table \ref{permuted}). We conjecture that there is a way to group permuted basement tableaux so that the weight in a group equals the weight of one multiline queue, see 
\cref{fig:queueTableaux} for an example.

To illustrate that our formulas are reasonable in terms of the 
number of terms,
Table \ref{permuted} records the number 
 of permuted basement tableaux
 (respectively, multiline queues) 
 in the Haglund-Haiman-Loehr formula (respectively our formula) for nonsymmetric Macdonald polynomials $E_{\lambda}$, 
 where $\lambda$ is a partition.
Note that for any composition $\mu$ whose parts rearrange to form $\lambda$,
the number of multiline queues that contribute to $F_{\mu}$ equals the number of multiline queues contributing to  
$F_{\lambda}$; similarly for the number of permuted basement tableaux contributing to the formula for the corresponding 
permuted basement Macdonald polynomial.

 \begin{table}[h]
\begin{tabular}{|p{4cm}|p{6cm}|p{4cm}|}
\noalign{\smallskip}
\hline
\noalign{\smallskip}
$\lambda$ & \# permuted basement tableaux & \# multiline queues  \\
\noalign{\smallskip}
\hline
\noalign{\smallskip}
	$(2, 1, 1, 0, 0)$ & $3$ & $3$   \\
\noalign{\smallskip}
\hline
\noalign{\smallskip}
	$(2,2,1,1,0,0)$ & $9$ & $7$  \\
\noalign{\smallskip}
\hline
\noalign{\smallskip}
	$(2,2,2,1,1,0,0)$ & $27$ & $13$  \\
\noalign{\smallskip}
\hline
\noalign{\smallskip}
	$(2,2,2,2,1,1,0,0)$ & $81$ & $21$  \\
\noalign{\smallskip}
\hline
\noalign{\smallskip}
	$(3,2,2,1,1,0,0)$ & $135$ & $105$  \\
\noalign{\smallskip}
\hline
\noalign{\smallskip}
	$(3,3,2,2,1,1,0,0)$ & $2025$ & $1029$  \\
\noalign{\smallskip}
\hline
\noalign{\smallskip}
	$(3,3,3,2,2,1,1,0,0)$ & $30375$ & $6643$  \\
\noalign{\smallskip}
\hline
\noalign{\smallskip}
	$(3,3,3,3,2,2,1,1,0,0)$ & $455625$ & $30723$  \\
\noalign{\smallskip}
\hline
\noalign{\smallskip}
	$(4, 3,3,3,2,2,1,1,0,0)$ & $3189375$ & $697515$  \\
\noalign{\smallskip}
\hline
\end{tabular}
\bigskip
\customlabel{permuted}{1}
\caption{A comparison of the number of terms in the Haglund-Haiman-Loehr formula versus 
	 our formula for $E_{\lambda}$. The first formula uses nonattacking fillings (which are a special 
	 case of permuted basement tableaux) and the second uses multiline queues.}
\end{table}

%
%
%
%
%
%
%
%
%
%
%


\section{A tableau version of multiline queues}\label{sec:tableau}
In this section we introduce some new \emph{queue tableaux} which are in bijection with multiline queues.
These tableaux  are
similar to the permuted basement tableaux, though the definitions of
attacking boxes,  coinversions, major index, and arm are all slightly
different.

Let $\mu=(\mu_1,\dots,\mu_n)$ be a composition with  $\mu_i\in \{0,1,\dots, k\}$.
The \emph{diagram} $D=D_{\mu}$ associated to $\mu$ is a
 sequence of $n$ columns of boxes where the $i$th column contains  $\mu_i$ boxes (justified to the bottom).
 Meanwhile the \emph{augmented diagram} $\widetilde{D} = \widetilde{D}_{\mu}$ is $D_{\mu}$
 augmented
 by a \emph{basement} consisting of $n$ boxes in a row just below these
 columns, see
	\cref{fig:attacking}.
 We number the rows of $\widetilde{D}$ from bottom to
top (starting from the basement in row $0$) and the columns from left to right (starting from column $1$). Abusing notation slightly,
we often use $D$ or $\widetilde{D}$ to refer to the collection of boxes in $D$ or $\widetilde{D}$.
We use $(i,j)$ to refer to the box  in column $i$ and row $j$ (if $\mu_i<j$ that box is empty).
For a box $x$, we denote by $d(x)$ the box directly below it.

Note that we will always be working with a diagram associated to a partition $\lambda$.

\begin{defn}\label{def:attacking}
	Let $D_{\lambda} $ be the diagram of shape $\lambda$, 
	and let $(i,j)\in D_{\lambda}$. The boxes attacking $(i,j)$ in the augmented diagram are 
	 (see \cref{fig:attacking} (a)):
\begin{itemize}
\item[(i.)] $(i',j)\in D_{\lambda}$ where $i\neq i'$,
\item[(ii.)] $(i',j-1) \in \widetilde{D}_{\lambda}$ where $i'>i$,
\item[(iii.)] $(i',j-1)\in \widetilde{D}_{\lambda}$ where $i'<i$ such that $\lambda_i=\lambda_{i'}$.

\end{itemize}
\end{defn}
Note that our definition of attacking boxes differs from that in \cite{HHL3,A} due to the third condition.

\begin{defn}
	\label{def:longest}
Let $\lambda=(\lambda_1,\dots,\lambda_n)$ be a partition and $\sigma\in S_n$ a permutation. We say $\sigma$ is \emph{compatible with} $\lambda$ if whenever $\lambda_i=\lambda_{i+1}$, we have that
	 $\sigma_{n-i}>\sigma_{n-i+1}$.
	Given a partition $\lambda$ and a
permutation $\sigma\in S_n$ that is compatible with $\lambda$, we say that
an \emph{augmented filling} of shape $\lambda$ and \emph{basement} $\sigma$ is
a filling of the boxes of $\widetilde{D}_{\lambda}$ with integers in $[n]$, where
the basement is filled from right to left with $\sigma_1,\ldots,\sigma_n$. %
\end{defn}

We use the notation $\phi: \widetilde{D}_{\lambda}\to [n]$ to denote an augmented filling.
Given a filling $\phi$, we say that a box $x$ is \emph{restricted} if the labels of $x$ and $d(x)$ are equal, i.e. if $\phi(d(x))=\phi(x)$, and \emph{unrestricted} otherwise.

Note that this definition of an augmented filling is consistent with the skyline fillings used in \cite{HHL3}; it is equivalent to
the definition of the same object in \cite{A}, though \cite{A} uses  English (rather than French) notation for diagrams.


\begin{figure}[!ht]
  \centerline{\includegraphics[height=1in]{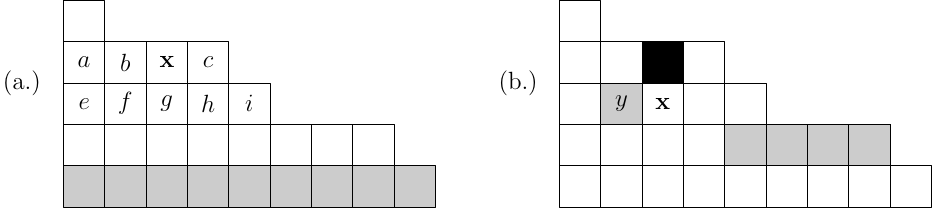}}
\centering
	\caption{(a) A tableau of shape $\lambda=(4,3,3,3,2,1,1,1,0)$ is shown, with the grey boxes representing the basement. The boxes attacking $x$ are: $a$, $b$, and $c$ (due to the first condition of \cref{def:attacking}), $h$ and $i$ (due to the second one), and $f$ (due to the third condition). The box $e$ is not attacking $x$, and $g=d(x)$. (b) The black box belongs to the leg and the grey boxes belong to the arm of $x$, with the box $y$ belonging to the arm provided that $y \neq d(y)$.}
\label{fig:attacking}
 \end{figure}

\begin{defn} \label{def:QT}
	Let $\lambda=(\lambda_1,\dots,\lambda_n)$ be a partition
	and let $\sigma=(\sigma_1,\dots, \sigma_n) \in S_n$ (written in one-line notation)
	be compatible with $\lambda$.
	 A \emph{queue tableau} of shape $\lambda$ with basement  $\sigma$ is an augmented filling $\phi:\widetilde{D}_{\lambda}\to [n]$ with basement $\sigma$ such that no two attacking boxes contain the same entry. Let $\QT_{\lambda}^{\sigma}$ denote the set of all queue tableaux of shape $\lambda$ with basement  $\sigma$.
\end{defn}

Note that due to the non-attacking condition, the entries in the basement must match the entries in row 1 directly above them, if they exist.

\begin{definition} \label{def:legarm}
Let $\lambda=(\lambda_1,\dots,\lambda_n)$ be a partition and let $\phi:\widetilde{D}_{\lambda}\to [n]$ be a queue tableau. 
Let $x=(i,j)$.

We define $\leg(x)=\lambda_i - j$ to be the number of boxes above $x$ in its column.

        The \emph{major index} is given by
        \[
        \maj(\phi)=\sum_{x\in D_{\lambda}\ :\ \phi(d(x))<\phi(x)} (\leg(x)+1).
        \]

        We define
        \begin{align*}
                \arm(x)&=\Big|\big\{(i',j-1) \in D_{\lambda}\ :\ i'>i,\ \lambda_{i'}<\lambda_i \big\} \Big| \\
                &+ \Big|\big\{(i',j) \in D_{\lambda}\ :\ i'<i,\ \lambda_{i'}=\lambda_i,\mbox{and}\ 
                (i',j) \text{ is unrestricted} \big\} \Big|
        \end{align*}
to be the number of boxes to the right of $x$ in the row below it, contained in columns shorter than its column, plus the number of unrestricted boxes to the left of and in the same row as $x$, contained in columns of the same length as $x$'s column.
\end{definition}

In \cref{fig:attacking} (b), the black box shows the leg of box $x$, while the grey boxes show the arm (assuming that none of the grey boxes to the left of $x$ are restricted).

	\begin{definition}\label{def:triple}
        A \emph{type $A$ quadruple} is a quadruple of boxes $\{x, d(x), y',y\}$ in $D_{\lambda}$ such that $y=d(y')$, the columns containing $x,y'$ are of the same length, and $x$ and $y'$ are in the same row. The two possible configurations for  type $A$ quadruples are shown in \cref{fig:inversion}.


A \emph{type $B$ triple} is a triple of boxes $\{x, d(x), y\}$ in $D_{\lambda}$ where $y$ is to the right of and in the same row as $d(x)$, and the column of $y$ is shorter than the column of $x$. See \cref{fig:inversion}.

		We say the triple or quadruple \emph{starts} at the cell $x$.

                A type $A$ quadruple is a \emph{coinversion} if all entries in its four cells are distinct, $\phi(x)>\phi(y')$, and either
        $\phi(x)<\phi(y)<\phi(d(x))$ or $\phi(y)<\phi(d(x))<\phi(x)$ or $\phi(d(x))<\phi(x)<\phi(y)$.   

        A type $B$ triple is a \emph{coinversion}
        if $\phi(y)<\phi(d(x))<\phi(x)$ or $\phi(d(x))<\phi(x)<\phi(y)$ or $\phi(x)<\phi(y)<\phi(d(x))$.

        We then define $\coinv(\phi)$ to be the number of coinversions
        coming from type $A$ quadruples and type $B$ triples, as shown in \cref{fig:inversion}.
\end{definition}

\begin{figure}[!ht]
  \centerline{\includegraphics[height=1.3in]{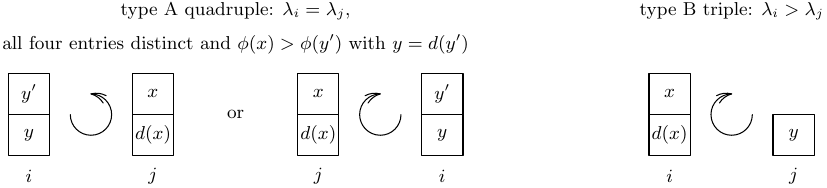}}
\centering
\caption{Quadruples and triple which are coinversions}
\label{fig:inversion}
 \end{figure}

\begin{defn}\label{def:weight}
        Let $\lambda=(\lambda_1,\dots,\lambda_n)$ be a partition and let $\phi:\widetilde{D}_{\lambda}\to [n]$ be a queue tableau. 
        The \emph{weight} of $\phi$ is
\begin{equation}\label{eq:weight}
        \wt(\phi)=q^{\maj(\phi)}t^{\coinv(\phi)}\prod_{x\in D_{\lambda}\ :\ \phi(d(x))\neq \phi(x)}
        \frac{1-t}{1-q^{\leg(x)+1}t^{\arm(x)+1}},
\end{equation}
We also define $x^\phi=\prod_{y\in D_{\lambda}}x_{\phi(y)}$ to be
the monomial in $x_1,\ldots,x_n$ where the power of $x_i$ is the number of boxes in $D_{\lambda}$ whose entry is $i$.
\end{defn}

The top line of \cref{fig:queueTableaux} shows the three queue tableaux of shape $\lambda=(2,2,1,0)$ with basement $(1,2,4,3)$, along with their weights.

\begin{remark}  \label{rem:pb}
	Let us compare our queue tableaux to the permuted basement
tableaux from \cite{A}.
To make the permuted basement tableaux from \cite{A} look more
like queue tableaux, we first reflect the tableaux from \cite{A}
from bottom to top, then rotate them $90^{\circ}$ counterclockwise.  Having done so, permuted basement tableaux which have shape $\alpha = (\alpha_1, \alpha_2, \dots, \alpha_n)$ (in the convention of \cite{A}) and basement $\sigma$ are the same as the queue tableaux from \cref{def:QT} of shape $\rev(\alpha) = (\alpha_n,\dots, \alpha_2, \alpha_1)$ and basement $\sigma$ except that the definition of attacking boxes for permuted basement tableaux only uses the first two conditions in \cref{def:attacking}.  \emph{All further definitions for permuted basement tableaux assume we have reflected and rotated the tableaux from \cite{A} as above.}

We again use coinversion triples to define \emph{coinversions}
for permuted basement tableaux.  Type $B$ coinversion triples are defined
 as in \cref{def:triple}.  However, for permuted basement tableaux,
a \emph{type $A'$ triple} is a triple of boxes $\{x, d(x),y\}$ in $D_{\lambda}$
where $y$ is to the left of and in the same row as $x$, and the column containing
$y$ is at most as long as the column containing $x$.  Such a triple is a \emph{type $A'$ coinversion triple} if $\phi(x) <\phi(y) < \phi(d(x))$ or $\phi(y)<\phi(d(x))<\phi(x)$ or $\phi(d(x))<\phi(x)<\phi(y)$.  We set $\coinv'(\phi)$
to be the total number of type $A'$ and $B$ coinversion triples.

The leg of a box is defined  as before, as is the major index $\maj(\phi)$.

Given a box $x$, we define $\arm'(x)$
to be the number of boxes in $D_{\lambda}$ to the right of $x$ in the row below it, contained in columns shorter than its column, plus the number of boxes to the left of and in the same row as $x$, contained in columns of length at most the length of $x$'s column.
If our shape is a partition,
the definition of $\arm'(x)$ agrees with the definition of $\arm(x)$ from \cref{def:legarm},
up to dropping the adjective ``unrestricted.'' If our shape is a partition
with distinct parts, the two definitions of arm agree.

Given all these definitions, the weight of a permuted basement tableau is
\begin{equation}\label{eq:weight2}
      \wt'(\phi)=q^{\maj(\phi)}t^{\coinv'(\phi)}\prod_{x\in D_{\lambda}\ :\ \phi(d(x))\neq \phi(x)}
        \frac{1-t}{1-q^{\leg(x)+1}t^{\arm'(x)+1}}.
\end{equation}
	(We note that there is a typo in \cite[(2)]{A}; the formula there has the product over boxes $u$ where $F(d(u))=F(u)$, but it should have $F(d(u))\neq F(u)$.)

Let $\mu=(\mu_1\dots,\mu_n)$ be a weak composition and $\sigma=(\sigma_1,\dots,\sigma_n)$ be a permutation.
Let $\PBT^{\sigma}_{\mu}$ denote the
 set of augmented fillings $\phi:\widetilde{D}_{\rev(\mu)}\to [n]$
 with basement $\sigma$
which are permuted basement tableaux.
Then the \emph{permuted basement Macdonald polynomial} is
	\begin{equation} \label{eq:pb}
E^{\sigma}_{\mu}(\x; q, t) = \sum_{\phi \in \PBT_{\mu}^{\sigma}} \wt'(\phi) x^{\phi}.
	\end{equation}
\end{remark}

\begin{remark}\label{rem:queue}
	Our queue tableaux are the same as permuted basement tableaux \cite{A, thesis}, and their weights agree,
	when $\lambda$ is a partition with distinct parts.
To see this, note that any non-attacking filling of a queue tableau is
automatically non-attacking as a filling of a permuted basement tableau. Moreover, when the parts of $\lambda$ are distinct, all non-attacking permuted basement fillings are also non-attacking according to \cref{def:attacking}, so the two sets of tableaux are equal.
Finally, note that when the parts of $\lambda$ are distinct, the definitions of arm agree on both sides; moreover,
there are no type $A$ quadruples or type $A'$ triples, so the coinversion statistics match
as well.
\end{remark}

        Recall  from \cref{def:Fmu} that
        $F_{\mu}$ is the generating function for multiline queues
        of type $\mu$.
\cref{prop:tableauxformula} below gives a tableau formula for
 $F_{\mu}$, and hence for the Macdonald polynomials
$P_{\lambda} = \sum_{\mu} F_{\mu}$, where the sum is over all distinct compositions $\mu$ obtained by permuting the parts of
$\lambda$.
This is the tableaux version of the
multiline queue formula from \cref{thm:main}.

\begin{thm}\label{prop:tableauxformula}
        Let $\mu=(\mu_1,\dots,\mu_n)$ be a weak composition, and
	let $\lambda:= \dec(\mu)$ be the partition obtained from $\mu$ by
	rearranging its parts in decreasing order.
        Choose $\sigma\in S_n$ to be the longest permutation 
        such that
        $\mu_{\sigma(1)} \leq \mu_{\sigma(2)} \leq
        \dots \leq \mu_{\sigma(n)}$ {(which implies that $\sigma$ is compatible with $\lambda$)}.
        We have that
\[
	F_{\mu} = \sum_{\phi\in\QT_{\dec(\mu)}^{\sigma}} \wt(\phi)x^\phi.
\]
\end{thm}

\begin{remark}
As mentioned earlier,
when $\lambda$ has distinct parts,  there are no type $A$ quadruples.
In this case the tableaux formula we obtain for Macdonald polynomials
(by combining
\cref{prop:tableauxformula} and
\cref{thm:main})
is essentially the one given by Lenart \cite{Lenart} (who gave a formula for $P_{\lambda}$ only in the case that $\lambda$ has distinct parts). To generalize that formula to arbitrary partitions, one needs the type $A$ quadruples.
\end{remark}

\begin{example}
Let us illustrate
\cref{prop:tableauxformula} for the case
$\mu = (0,1,2,2)$. Using the notation of that theorem,
we have $\lambda = (2,2,1,0)$ and $\sigma = (1,2,4,3)$,
so we can compute $F_{\mu}$ not only by summing over
the multiline queues of type $\mu$, but also by summing over the
queue tableaux in $\QT_{\lambda}^{\sigma}$, that is, the queue tableaux
of shape $(2,2,1,0)$ with basement $\sigma = (1,2,4,3)$
(read from right to left).
This is shown in the top line of
 \cref{fig:queueTableaux}.

Meanwhile, we know from \cref{prop:PB} that
$F_{\mu}=E_{(0,1,2,2)}^{\sigma}$.  So we can also compute
$F_{\mu}$ using \eqref{eq:pb} as the sum over permuted basement tableaux which
are augmented fillings of $\widetilde{D}_{(2,2,1,0)}$ with
basement $\sigma$.
This is shown in the second line of \cref{fig:queueTableaux}.

Note that the sum of the weights of the queue tableaux is the same as the sum of the weights of the permuted basement tableaux;   in particular, the sum of the weights of the third and fourth permuted basement tableaux equals the weight of the third queue tableau.
\end{example}

\begin{figure}[!ht]
  \centerline{\includegraphics[width=\textwidth]{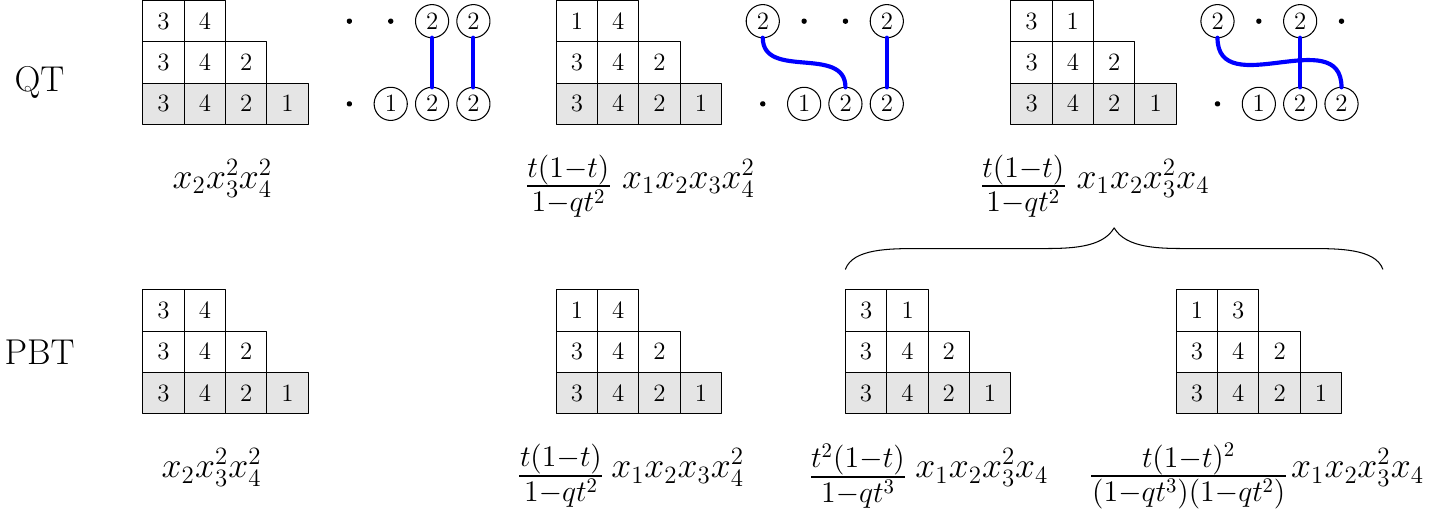}}
\centering
	\caption{The top line shows the three queue tableaux of shape
        $\lambda=(2,2,1,0)$ and basement $\sigma=(1,2,4,3)$ with the corresponding
        multiline queues of type $\mu=(0,1,2,2)$.  The second line shows
	the four permuted basement tableaux of shape $\rev(\lambda) = (0,1,2,2)$ and basement $(1,2,4,3)$.
        Note that these correspond to the permuted basement tableaux from  \cite{A}
        but we are displaying them differently: to get ours from his, we
        first reflect his tableaux from bottom to top, then rotate them
        $90^{\circ}$ counterclockwise.
        The total weight for both is $x_2x_3^2x_4^2+(x_1x_2x_3^2x_4+x_1x_2x_3x_4^2)\frac{t(1-t)}{1-q t^2}$.}
\label{fig:queueTableaux}
 \end{figure}

To prove \cref{prop:tableauxformula}, we show that there is a direct weight-preserving bijection between
$\MLQ(\mu)$
and $\QT_{\lambda}^{\sigma}$
	where  $\lambda$ is the partition obtained from $\mu$ by rearranging its parts in decreasing order,
	and
	 $\sigma\in S_n$ is the longest permutation
	 such that
	$\mu_{\sigma(1)} \leq
	\dots \leq \mu_{\sigma(n)}$.
Our bijection is the following.

	\begin{definition}\label{def:Tab}
		Suppose $\mu$ is a composition with maximal entry $L$
		and let $Q\in\MLQ(\mu)$.
	Choose $\lambda$ and $\sigma$ as in
\cref{prop:tableauxformula}.
	Let $\phi$ be an augmented filling of shape $\lambda$ with basement labeled by $\sigma$ from right to left. Let $i_1,\ldots,i_{A}$ be the columns containing the string of linked balls of label $A$ $(1\leq A\leq L)$ that begin from the ball at column $i_1$ in row 1 of $Q$. Label the boxes of the column of $\phi$ with $i_1$ in the basement by $i_1,\ldots,i_A$ from bottom to top. Let $\Tab(Q)$ denote the resulting tableau.
	\end{definition}
In \cref{fig:queuetableau} we give the tableau  corresponding to the
multiline queue in \cref{fig:MLQ_example}.

\begin{figure}[!ht]
  \centerline{\includegraphics[height=1in]{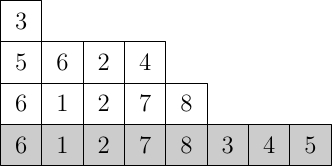}}
\centering
	\caption{$\Tab(Q)$ is shown, where $Q$ is
	the multiline queue in \cref{fig:MLQ_example}, with a description of its statistics in \cref{ex:TabQ}.}
\label{fig:queuetableau}
 \end{figure}

 \begin{lemma}\label{lem:bij}
Let $Q\in\MLQ(\mu)$,
	and choose $\lambda$ and $\sigma$ as in \cref{prop:tableauxformula}.
	Then $\Tab(Q)\in\QT_{\lambda}^{\sigma}$. Moreover, this map is a bijection from
	$\MLQ(\mu)$ to $\QT_{\lambda}^{\sigma}$.
\end{lemma}

\begin{proof}
We claim that the filling $\Tab(Q)$ obtained from an MLQ in this way is non-attacking.

First, if a ball labeled $j$ is directly above a ball labeled $i$ in $Q$ in row $r$ and column $c$, then either $j<i$, or $j=i$ in which case the two balls are paired. There are two boxes in $\Tab(Q)$ containing the label $c$ in rows $r$ and $r+1$ respectively. If $j<i$, the box in row $r+1$ of $\Tab(Q)$ is to the right of the box in row $r$ since all columns corresponding to label $j$ are by construction to the right of all columns corresponding to label $i$. If $j=i$,  both boxes labeled $c$ are in the same column, and thus non-attacking in $\Tab(Q)$ in both cases.

There is a simple map from a filling $\phi\in\QT_{\lambda}^{\sigma}$ to $\MLQ(\mu)$.
Note that if $\lambda$ and $\sigma$ are
as in \cref{prop:tableauxformula}, then $\mu$ is determined from them by
	$\mu_i = \lambda_{n+1-\sigma^{-1}(i)}$.
	Let $M$ be a multiline queue with $\lambda_1$ rows and $n$ columns. For each $j$, suppose the entries in column $j$ of $\phi$ are, from bottom to top, $x_0, x_1,\ldots,x_{\lambda_j}$, with $x_0=\sigma(n+1-j)$ denoting the
entry in the basement of column $j$.  If $\lambda_j \geq 1$,
then for $1 \leq i \leq \lambda_j$, we place a ball in
row $i$ and column $x_i$  in the multiline queue $Q$,
connecting pairs of these $\lambda_j$ balls to each other when they
	occur in successive rows,
and giving them all label $\lambda_j$.  In particular
the bottom ball has label $\lambda_j$ and occurs in column
$x_1 = x_0= \sigma(n+1-j)$.  If $\lambda_j = 0$ then we do not
add any balls and so we get an empty spot (or $0$)
in column $\sigma(n+1-j)$ of the bottom row.
Therefore for all $j$, the bottom row of $M$ contains $\lambda_j$
in column $\sigma(n+1-j)$, or equivalently,
column $i$ of the bottom row contains
$\lambda_{n+1-\sigma^{-1}(i)}$.  In particular $M$ has type $\mu$.

It is easy to check that this map reverses our construction $\Tab$ in \cref{def:Tab}.
\end{proof}

 \begin{lemma}\label{lem:stats}
 Let $Q$ be a multiline queue and $\phi=\Tab(Q)$ its corresponding queue tableau.
\begin{enumerate}
\item Let $x$ be a cell in row $r$ and a column of length $j$ of $\Tab(Q)$. Then $\leg(x)+1=j-r+1$ and $j$ is the label of the ball corresponding to $x$ in $Q$. 
\item If for a cell $x \in \Tab(Q)$, $\phi(d(x))<\phi(x)$, let $b(x)$ and $b(d(x))$ be the balls in $Q$ corresponding to $x$ and $d(x)$. The ball pairing $b(x)$ and $b(d(x))$ is wrapping, and $\leg(x)+1=j-r+1$ where $j$ is the label of both balls, and $r$ is the row containing $b(x)$ in $Q$. Thus $\maj(\Tab(Q))$ is equal to the power of $q$ in the numerator of $\wt(Q)$.
\item Let $U(r,j)$ be the set of unrestricted boxes in row $r$ and columns of length $j$ of $\Tab(Q)$. Then the contribution
\[
\prod_{x\in U(r,j)} \frac{1-t}{1-q^{\leg(x)+1}t^{\arm(x)+1}}
\]
matches the analogous contribution of the ball pairings starting from a ball labeled $j$ in row $r$ of the multiline queue.
\item The coinversions
of type $B$ in $\Tab(Q)$ count the number of balls skipped of lower labels in $Q$. The coinversions of type $A$ count the number of balls skipped of the same label in $Q$.
\end{enumerate}
\end{lemma}

\begin{proof}
(1), (2), and (4) are immediate from the definitions.

For (3), fix $2 \leq j \leq \lambda_1$. Let $U(r,j)=\{x_1,\ldots,x_k\}$ wbe the set of unrestricted boxes contained in columns of size $j$ in row $r$ of $\Tab(Q)$. Here we suppose that for all $i$, $x_i$ is to the left of $x_{i+1}$. Let $\ell_j$ be the number of cells in columns of length smaller than $j$ in row $r-1$. Then $\arm(x_i)+1=\ell_j+i$ and $\leg(x_i)+1 = j-r+1$, and so we get the contribution
\begin{equation}\label{eq:arms}
\prod_{x\in U(r,j)} \frac{1-t}{1-q^{\leg(x)+1}t^{\arm(x)+1}} 
= \prod_{i=1}^k \frac{1-t}{1-q^{j-r+1}t^{\ell_j+i}}
\end{equation}
to the weight of $\Tab(Q)$ for the entries $U(r,j)$.

On the other hand, each $x_i\in U(r,j)$ corresponds to a ball with label $j$ in row $r$ of $Q$ that is not trivially paired. There are also $\ell_j$ balls of labels smaller than $j$ in row $r-1$ of $Q$. Then the set of the numbers of free balls in row $r-1$ before the pairing of each ball corresponding to $x_i\in U(r,j)$ is precisely $\{\ell_j+1,\ell_j+2,\ldots,\ell_j+k\}$. Since every ball corresponding to $x_i\in U(r,j)$ contributes a factor of $(1-t)/(1-q^{j-r+1}t^{\# \free})$, these contributions to $\wt(Q)$ match the contributions in \eqref{eq:arms} to $\wt(\Tab(Q))$.

Finally, comparing \cref{def:wt} to
\cref{def:weight}, we see that $\wt(Q) = \wt(\Tab(Q))$.
\end{proof}

\begin{example}\label{ex:TabQ}
We compare the weights of the pairings of balls in the multiline queue $Q$ from \cref{fig:MLQ_example} to the statistics of the corresponding queue tableau $\phi=\Tab(Q)$ in \cref{fig:queuetableau}.

\begin{itemize}[leftmargin=*]
\item In $Q$, one ball is skipped in the pairing of balls labeled
$3$ between row $3$ and $2$. This pairing corresponds to the coinversion starting at the cell $u=(1,3)$ in $\Tab(Q)$, which is the type $B$ triple \raisebox{-2pt}{$\qtrip{$3$}{$5$}{$4$}$}\ . The total weight of pairings from row 3 to row 2 in $Q$ is $t(1-t)/(1-q t^4)$. In $\Tab(Q)$ the quantity $t(1-t)/(1-qt^4)$ comes from cell $u$, with $\arm(u)+1=4$ and $\leg(u)+1=1$.
\item In $Q$, no balls are skipped in the pairing of balls labeled $3$ between rows 2 and 1. Accordingly, there are no coinversions starting at the corresponding cell $u=(1,2)$ in $\Tab(Q)$. The total weight of this pairing in $Q$ is $(1-t)/(1-q^2t^5)$. This is consistent with the contribution to $\wt(\phi)$ from $u$ with $\arm(u)+1=5$ and $\leg(u)+1=2$.
\item In $Q$, the nontrivial pairings of balls
labeled $2$ (from row 2 to row 1) skip two and zero balls, respectively. The first of these pairings corresponds to two coinversions starting at the cell $u_1=(2,2)$ in $\Tab(Q)$ contributing the weight $t^2$: the type $A$ quadruple \raisebox{-2pt}{\qqua{$6$}{$1$}{$4$}{$7$}} and the type $B$ triple \raisebox{-2pt}{\qtrip{$6$}{$1$}{$8$}}. The second pairing corresponds to the cell $u_2=(4,2)$ in $\Tab(Q)$, which has no coinversions starting from it.
\item In $Q$, the weights of the nontrivial pairings of balls
labeled $2$ (from row 2 to row 1) are $qt^2(1-t)/(1-qt^3)$ and $(1-t)/(1-qt^2)$. The cell $u_1$ in $\Tab(Q)$ has $\arm(u_1)+1=2$ and $\leg(u_1)+1=1$, and has $\phi(d(u_1))<\phi(u_1)$, accounting for the additional weight $q(1-t)/(1-qt^2)$. The second pairing corresponds to the contribution from the cell $u_2=(4,2)$, which has $\arm(u_2)+1=3$ and $\leg(u_2)+1=1$, accounting for the additional weight $(1-t)/(1-qt^3)$; the products of these contribute equally in $\wt(Q)$ and $\wt(\phi)$, respectively.

\end{itemize}

\end{example}

\begin{cor}
	\label{thm:bijection}
Let $\mu$ be a partition, and choose $\lambda$ and $\sigma$ as in \cref{prop:tableauxformula}. Then the bijection
$\Tab: MLQ(\mu) \rightarrow \QT_{\lambda}^{\sigma}$ is weight-preserving.
\end{cor}
\begin{proof}
	This follows from
	 \cref{lem:bij}
and \cref{lem:stats}.
\end{proof}

\begin{proof}[Proof of \cref{prop:tableauxformula}]
This follows immediately from \cref{thm:bijection} and \cref{def:Fmu}.
\end{proof}




\begin{remark}
There is an alternative notion of type $A$ quadruple and coinversion for which \cref{prop:tableauxformula} holds. Define a \emph{type $A''$ quadruple} to be a quadruple of boxes\linebreak $\{x,d(x),y,y'\}$ in $D_{\lambda}$ where $x,y'$ are in the same row and in columns $j<i$, respectively, $\lambda_j=\lambda_i$, and $d(y')=y$. We say that this type $A''$ quadruple is a \emph{coinversion} if all four entries in the cells $\{x,d(x),y,y'\}$ are distinct, and either $\phi(x)<\phi(y)<\phi(d(x))$ or $\phi(y)<\phi(d(x))<\phi(x)$ or $\phi(d(x))<\phi(x)<\phi(y)$ {(the entries in the cells $\{x,d(x),y\}$ are cyclically increasing when read in clockwise order)}.

Then we can define $\coinv''(\phi)$ to be the number of type $B$ and type $A''$ coinversions, and replace $\coinv(\phi)$ by $\coinv''(\phi)$ in the formula for $\wt(\phi)$ in \eqref{eq:weight}.

This equivalence of weights is due to \cref{lem:order}. We recall the correspondence briefly. We think of columns of the same height in the queue tableaux as balls with the same label in the multiline queue. We think of coinversions in the queue tableaux as skipped balls in the multiline queue, and in particular, we think of type $A$ quadruples as skipped balls of the same label. In the multiline queue, the weight of each pairing is dependent on the pairing order of balls of the same label. The condition $\phi(x)>\phi(y')$ (from \cref{fig:inversion}) for type $A$ quadruples corresponds to a right-to-left pairing order in the multiline queue. On the other hand, type $A''$ quadruples correspond to another pairing order, that is determined by the entries in the row containing the cells $x,y'$. From \cref{lem:order}, we have that the total weight summed over all multiline queues is independent of the pairing order, from which we conclude that using $\coinv''$ gives the same total weight after summing over all tableaux.
\end{remark}

\bibliographystyle{alpha}
\bibliography{bibliography}

\newcommand{\etalchar}[1]{$^{#1}$}
\begin{thebibliography}{LNS{\etalchar{+}}17b}

\bibitem[AAMP12]{AritaAyyerMallickProlhac12}
Chikashi Arita, Arvind Ayyer, Kirone Mallick, and Sylvain Prolhac.
\newblock Generalized matrix ansatz in the multispecies exclusion process---the
  partially asymmetric case.
\newblock {\em J. Phys. A}, 45(19):195001, 16, 2012.

\bibitem[AGS18]{AasGrinbergS}
Erik Aas, Darij Grinberg, and Travis Scrimshaw.
\newblock Multiline queues with spectral parameters.
\newblock 2018.
\newblock arXiv:1810.08157.

\bibitem[Ale16]{A}
Per Alexandersson.
\newblock Non-symmetric {M}acdonald polynomials and {D}emazure-{L}usztig
  operators.
\newblock arXiv:1602.05153, 2016.

\bibitem[CdGW]{CGW-arxiv}
Luigi Cantini, Jan de~Gier, and Michael Wheeler.
\newblock Matrix product and sum rule for {M}acdonald polynomials.
\newblock FPSAC abstract.

\bibitem[CdGW15]{CGW}
Luigi Cantini, Jan de~Gier, and Michael Wheeler.
\newblock Matrix product formula for {M}acdonald polynomials.
\newblock {\em J. Phys. A}, 48(38):384001, 25, 2015.

\bibitem[Che91]{Cherednik1}
Ivan Cherednik.
\newblock A unification of {K}nizhnik-{Z}amolodchikov and {D}unkl operators via
  affine {H}ecke algebras.
\newblock {\em Invent. Math.}, 106(2):411--431, 1991.

\bibitem[Che94]{Cherednik2}
Ivan Cherednik.
\newblock Integration of quantum many-body problems by affine
  {K}nizhnik-{Z}amolodchikov equations.
\newblock {\em Adv. Math.}, 106(1):65--95, 1994.

\bibitem[Che95]{Cher1}
Ivan Cherednik.
\newblock Nonsymmetric {M}acdonald polynomials.
\newblock {\em Internat. Math. Res. Notices}, (10):483--515, 1995.

\bibitem[DEHP93]{DEHP}
B.~Derrida, M.~R. Evans, V.~Hakim, and V.~Pasquier.
\newblock Exact solution of a {$1$}{D} asymmetric exclusion model using a
  matrix formulation.
\newblock {\em J. Phys. A}, 26(7):1493--1517, 1993.

\bibitem[EFM09]{EvansFerrariMallick08}
Martin~R. Evans, Pablo~A. Ferrari, and Kirone Mallick.
\newblock Matrix representation of the stationary measure for the multispecies
  {TASEP}.
\newblock {\em J. Stat. Phys.}, 135(2):217--239, 2009.

\bibitem[Fer]{thesis}
Jeffrey~Paul Ferreira.
\newblock Row-strict quasisymmetric schur functions, characterizations of
  demazure atoms, and permuted basement nonsymmetric {M}acdonald polynomials.
\newblock Ph.D. thesis, University of California, Davis, 2011.

\bibitem[FM07]{FerrariMartin}
Pablo~A. Ferrari and James~B. Martin.
\newblock Stationary distributions of multi-type totally asymmetric exclusion
  processes.
\newblock {\em Ann. Probab.}, 35(3):807--832, 2007.

\bibitem[Hai01]{HaimanHilbert}
Mark Haiman.
\newblock Hilbert schemes, polygraphs and the {M}acdonald positivity
  conjecture.
\newblock {\em J. Amer. Math. Soc.}, 14(4):941--1006, 2001.

\bibitem[Hai06]{HaimanICM}
Mark Haiman.
\newblock Cherednik algebras, {M}acdonald polynomials and combinatorics.
\newblock In {\em International {C}ongress of {M}athematicians. {V}ol. {III}},
  pages 843--872. Eur. Math. Soc., Z\"urich, 2006.

\bibitem[HHL05a]{HHL1}
J.~Haglund, M.~Haiman, and N.~Loehr.
\newblock A combinatorial formula for {M}acdonald polynomials.
\newblock {\em J. Amer. Math. Soc.}, 18(3):735--761, 2005.

\bibitem[HHL05b]{HHL2}
J.~Haglund, M.~Haiman, and N.~Loehr.
\newblock Combinatorial theory of {M}acdonald polynomials. {I}. {P}roof of
  {H}aglund's formula.
\newblock {\em Proc. Natl. Acad. Sci. USA}, 102(8):2690--2696, 2005.

\bibitem[HHL08]{HHL3}
J.~Haglund, M.~Haiman, and N.~Loehr.
\newblock A combinatorial formula for nonsymmetric {M}acdonald polynomials.
\newblock {\em Amer. J. Math.}, 130(2):359--383, 2008.

\bibitem[KMO15]{KMO15}
Atsuo Kuniba, Shouya Maruyama, and Masato Okado.
\newblock Multispecies {TASEP} and combinatorial {$R$}.
\newblock {\em J. Phys. A}, 48(34):34FT02, 19, 2015.

\bibitem[KT07]{KasataniTakeyama}
Masahiro Kasatani and Yoshihiro Takeyama.
\newblock The quantum {K}nizhnik-{Z}amolodchikov equation and non-symmetric
  {M}acdonald polynomials.
\newblock {\em Funkcial. Ekvac.}, 50(3):491--509, 2007.

\bibitem[Len09]{Lenart}
Cristian Lenart.
\newblock On combinatorial formulas for {M}acdonald polynomials.
\newblock {\em Adv. Math.}, 220(1):324--340, 2009.

\bibitem[LNS{\etalchar{+}}17a]{Lenart2}
C.~Lenart, S.~Naito, D.~Sagaki, A.~Schilling, and M.~Shimozono.
\newblock A uniform model for {K}irillov-{R}eshetikhin crystals {III}:
  nonsymmetric {M}acdonald polynomials at {$t=0$} and {D}emazure characters.
\newblock {\em Transform. Groups}, 22(4):1041--1079, 2017.

\bibitem[LNS{\etalchar{+}}17b]{Lenart1}
Cristian Lenart, Satoshi Naito, Daisuke Sagaki, Anne Schilling, and Mark
  Shimozono.
\newblock A uniform model for {K}irillov-{R}eshetikhin crystals {II}. {A}lcove
  model, path model, and {$P=X$}.
\newblock {\em Int. Math. Res. Not. IMRN}, (14):4259--4319, 2017.

\bibitem[Mac95]{Macdonald}
I.~G. Macdonald.
\newblock {\em Symmetric functions and {H}all polynomials}.
\newblock Oxford Mathematical Monographs. The Clarendon Press, Oxford
  University Press, New York, second edition, 1995.
\newblock With contributions by A. Zelevinsky, Oxford Science Publications.

\bibitem[Mac96]{MacdonaldBourbaki}
I.~G. Macdonald.
\newblock Affine {H}ecke algebras and orthogonal polynomials.
\newblock {\em Ast\'erisque}, (237):Exp. No. 797, 4, 189--207, 1996.
\newblock S\'eminaire Bourbaki, Vol. 1994/95.

\bibitem[Mac03]{MacdonaldAffine}
I.~G. Macdonald.
\newblock {\em Affine {H}ecke algebras and orthogonal polynomials}, volume 157
  of {\em Cambridge Tracts in Mathematics}.
\newblock Cambridge University Press, Cambridge, 2003.

\bibitem[Mar]{Marshall}
Dan Marshall.
\newblock Symmetric and nonsymmetric {M}acdonald polynomials.

\bibitem[Mar18]{Martin}
James~B. Martin.
\newblock Stationary distributions of the multi-type {A}{S}{E}{P}s.
\newblock 2018.
\newblock arXiv:1810.10650.

\bibitem[MGP68]{bio}
J~Macdonald, J~Gibbs, and A~Pipkin.
\newblock Kinetics of biopolymerization on nucleic acid templates.
\newblock {\em Biopolymers}, 6, 1968.

\bibitem[PEM09]{ProlhacEvansMallick09}
S.~Prolhac, M.~R. Evans, and K.~Mallick.
\newblock The matrix product solution of the multispecies partially asymmetric
  exclusion process.
\newblock {\em J. Phys. A}, 42(16):165004, 25, 2009.

\bibitem[RY11]{RamYip}
Arun Ram and Martha Yip.
\newblock A combinatorial formula for {M}acdonald polynomials.
\newblock {\em Adv. Math.}, 226(1):309--331, 2011.

\bibitem[Sch06]{Schwer}
Christoph Schwer.
\newblock Galleries, {H}all-{L}ittlewood polynomials, and structure constants
  of the spherical {H}ecke algebra.
\newblock {\em Int. Math. Res. Not.}, pages Art. ID 75395, 31, 2006.

\bibitem[Spi70]{Spitzer}
Frank Spitzer.
\newblock Interaction of {M}arkov processes.
\newblock {\em Advances in Math.}, 5:246--290 (1970), 1970.

\end{thebibliography}

\end{document}